\documentclass[A4paper,12pt,twoside,reqno,psamsfonts]{amsart}

\usepackage{amssymb,amsmath,amsfonts}
\usepackage{ae,aecompl}

\RequirePackage[dvips]{graphicx}
\usepackage{psfrag}
\usepackage{comment}
\usepackage[usenames]{color}
\usepackage[dvips]{graphicx} 

\usepackage[latin1]{inputenc}
\usepackage[T1]{fontenc}
\usepackage{graphicx}
\usepackage{enumerate}	
\usepackage{mathrsfs}
\usepackage{epstopdf}
\usepackage{xcolor}

\usepackage[linktocpage=true]{hyperref}

\def\?{|}
\def\ldim{\underline{\dim}}
\def\udim{\overline{\dim}}

\DeclareMathOperator\dist{dist}

\newtheorem{theorem}{Theorem}[section]
\newtheorem{lemma}[theorem]{Lemma}

\newtheorem{proposition}[theorem]{Proposition}
\newtheorem{question}[theorem]{Question}
\theoremstyle{definition}

\newtheorem{example}[theorem]{Example}
\theoremstyle{remark}
\newtheorem{remark}[theorem]{Remark}

\numberwithin{equation}{section}

\newcommand{\C}{\mathcal{C}}
\newcommand{\D}{\mathcal{D}}
\newcommand{\E}{\mathcal{E}}
\newcommand{\G}{\mathcal{G}}

\newcommand{\I}{\mathcal{I}}
\newcommand{\K}{\mathcal{K}}

\newcommand{\EE}{\mathbb{E}}
\newcommand{\NN}{\mathbb{N}}
\newcommand{\RR}{\mathbb{R}}
\newcommand{\PP}{\mathbb{P}}

\DeclareMathOperator*{\g}{\gamma}

\title{ A class of random  Cantor sets}

\author{Changhao Chen}
\address{Chanhao Chen\\
Department of Mathematical Sciences, P.O. Box 3000, 90014
University of Oulu, Finland}
\email{changhao.chenm@gmail.com}



\thanks{The author acknowledge the support of the Vilho, Yrj\"o, and Kalle V\"ais\"al\"a foundation.}

\begin{document}

\maketitle

\begin{abstract} 
In this paper we study a class of random Cantor sets. We determine their almost sure Hausdorff, packing, box, and Assouad dimensions. From a topological point of view, we also compute their typical dimensions in the sense of Baire category. For the natural random measures on these random Cantor sets, we consider their almost sure  lower and upper local dimensions. In the end we study the hitting probabilities of a special subclass of these random Cantor sets. 
\end{abstract}



\section{Introduction}


In this paper we consider a class of random Cantor sets. This consists of a sample space $\Omega$ and a probability measure $\PP$. The sample space $\Omega$ contains a family of compact subsets of $[0,1]^{d}$, furthermore $\Omega$ is a compact metric space endowed with the Hausdorff metric. We will compute their almost sure and typical  dimensions. For each object of $\Omega$, we put a natural measure on this object. We also calculate the local dimensions of these natural measures. In the end, we study the hitting probabilities of a special subclass of these random Cantor sets. We start by a description of these random Cantor sets. Closely related random models have been considered in \cite{Chen2015, ChenKoivusalo, Shmerkin3, Shmerkin1, Shmerkin2}.

\subsection{Random Cantor sets}

Let $\{M_k\}_{k\geq 1}$ and $\{N_k\}_{k\geq 1}$ be sequences of integers with $1\leq N_k\leq M_k^d, M_k\ge 2$ for all $k$. Let
\begin{equation}\label{eq:P_n}
 P_{n}=\prod_{i=1}^{n} N_{i},~ r_{n}=\left(\prod_{i=1}^{n} M_{i}\right)^{-1}
 \end{equation}  
We denote by $\D_n=\D_n([0,1]^d)$ the family of  $r_n^{-1}$-adic closed subcubes of $[0,1]^d$,
\[
\D_n=\Big\{\prod^d_{\ell=1}[i_{\ell}r_n, (i_{\ell}+1)r_n]:  0\leq i_{\ell}\leq r_n^{-1}-1 \Big\}
\]
and let $\D= \bigcup_{n\in \NN}\D_n.$ We divide the unit cube $[0,1]^d$ into $M_1^d$ interior disjoint $M_1$-adic closed subcubes and randomly choose interior disjoint $N_1\leq M_1^d$ of these closed subcubes in the following way. We randomly choose a cube among $M_1^d$ cubes uniformly which means that every cube has the same probability of being chosen, then we randomly choose an other cube among the remaining $M_1^d-1$ cubes uniformly, and continue this process until we obtain $N_1$ cubes. Note that  each of the closed subcubes has the same probability (i.e.$ N_1/ M_1^{d}$) of being chosen, and denote their union by $E_1$. Given $E_{n}$, a random collection of $P_n$ interior disjoint $r_n^{-1}$ - adic closed subcubes of $[0,1]^{d}.$ For each cube of $E_n$, we divide it  into $M_{n+1}^d$ interior disjoint $r_{n+1}^{-1}$-adic closed subcubes and randomly choose interior disjoint $N_{n+1}$ of these closed subcubes in the same fashion as above (i.e. we randomly choose a cube among $M_{n+1}^d$ cubes uniformly, then we randomly choose an other cube among the remaining $M_{n+1}^d-1$ cubes uniformly, and continue this process until we obtain $N_{n+1}$ cubes). We ask that the choices are independent for different cubes of $E_n$. Let $E_{n+1}$ be the union of the chosen closed cubes and 
\[
E^{\omega}=E=\bigcap_{n=1}^\infty E_n
\] 
be a random limit set. Let $\Omega=\Omega(M_k,N_k)$ be our probability space which consists of  all the possible outcomes of random limit sets. For convenience we will write $E\in \Omega, \omega \in \Omega,$ or $E^{\omega} \in \Omega$ in the following. Our main object of study in this paper is the space $\Omega$. Figure \ref{figure1} shows an example of this construction.

\begin{figure}
   \centering
\resizebox{0.9\textwidth}{!}{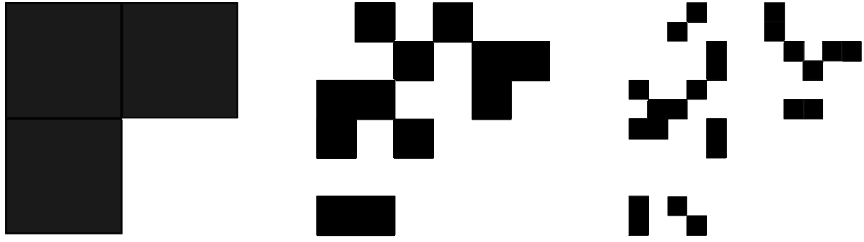}
\caption{The first three steps in the construction of $E$ with $M_1=2, N_1=3, M_2=3, N_2=4, M_3=2, N_3=2$.} 
\label{figure1} 
\end{figure}

\subsection{The topological approach}

Let $\K=\K([0,1]^{d})$ be all the compact subsets of unit cube $[0,1]^{d}$. We endow $\K$ with the Hausdorff metric. Recall that the Hausdorff distance of two compact sets $E$ and $F$ of $\K$ is defined by
\[
 d_H(E,F)=\inf\{\varepsilon>0: E \subset F^\varepsilon\text{ and }F\subset E^\varepsilon\},
\]
where $E^\varepsilon=\{x\in \RR^{d}:\dist(x,E)<\varepsilon\}$. Observe that $\Omega=\Omega(M_k,N_k) \subset \K$ and $\Omega$ is a closed subset of $\K.$ Together with the well known fact that $\K$ is a compact space, we obtain that $\Omega$ is compact subset of $\K$. 

Recall that a subset of a metric space $X$ is of first category if it is a countable union of \emph{nowhere dense} sets (i.e. whose closure in $X$ has empty interior); otherwise it is called of second category.
We say that a \emph{typical} element $x \in X$ has property $P$, if the complement of 
\[
 \{x \in X : x \text{ satisfies } P\}
\] 
is of first category. For the basic properties and various applications of Baire Category, we refer to 
\cite{Oxtoby, Stein}.

\subsection{Dimension and measure}

Let $E \subset [0,1]^d$. For any $s\geq 0$, the $s$-dimensional Hausdorff measure is defined as 
$\mathcal{H}^s(E)=\lim_{\delta \rightarrow 0}\mathcal{H}^s_\delta(E)$ where 
\[
\mathcal{H}^s_\delta(E)= \inf \Big\{\sum^{\infty}_{n=1} \?U_n\?^s: E\subset \bigcup^{\infty}_{n=1} U_n, \?U_n\? \leq \delta, n \in \NN  \Big\},
\]
and $\?U\?$ is the diameter of $U$. The Hausdorff dimension of $E$ is 
\[
\dim_H E=\sup \{s\geq 0: \mathcal{H}^s(E)=\infty\} =\inf \{s \geq 0: \mathcal{H}^s(E)=0\}.
\]

For any $\delta>0$, let $\mathcal{N}(E, \delta)$ be
 the smallest
number of sets of diameter at most $\delta$ which can cover $E$. Then the lower and upper box dimensions are defined respectively as 
\[
\ldim_B E = \liminf_{\delta\rightarrow 0}\frac{\log \mathcal{N} (E,\delta)}{- \log \delta}, \,\,\udim_B E = \limsup_{\delta\rightarrow 0}\frac{\log\mathcal{N} (E,\delta)}{- \log \delta}.
\] 
If $\ldim_B E=\udim_B E$ we denote this common value by $\dim_B E$ and call it the box dimension of $E$. 

The packing dimension of $E$ is defined as  
\[
\dim_PE=\inf \Big\{\sup \udim_B F_n : F= \bigcup_{n=1}^{\infty} F_n \Big\}.
\]

The Assouad dimension of $E$ is defined as
\begin{align*}
\dim_{A} E=\inf \{s \geq 0  : \exists~  C>0  &\text{ s.t. }  \forall ~ 0<r<R\leq \sqrt{d},   \\
& \sup_{x \in E}\mathcal{N}\left(E \cap B(x, R),r \right) \leq C \left(\frac{R}{r} \right)^s\}.
\end{align*}


The basic relationships of these dimensions are 
\[
\dim_H E\leq \ldim_B E, \,\, \dim_P E \leq \udim_B E \leq \dim_A E.
\]
For more details and further properties of these dimensions, we refer to \cite{Falconer2003, Mattila1995} and especially \cite{Luukkainen} for the Assouad dimension.

Let $\nu$ be a Radon measure on $\RR^d$. For $x \in \RR^d$, the lower and upper local (pointwise) dimensions of $\nu$ at $x$ are defined respectively as 
\[
\ldim(\nu,x)=\liminf_{r\rightarrow 0} \frac{\log \nu(B(x,r))}{\log r}, \,\,
\udim(\nu,x)=\limsup_{r\rightarrow 0} \frac{\log \nu(B(x,r))}{\log r}.
\]
If $\ldim(\nu,x)=\udim(\nu,x)$ we denote this common value by $\dim(\nu,x)$, and call it the local dimension of  $\nu$ at $x$. For further details and basic properties on the local dimensions of measures, see \cite[Chapter 10]{Falconer1997}.

We consider the natural random measure on the random Cantor set. Let $E=\bigcap_{n=1}^\infty E_n$ be a realization. For each $n\in \NN$, let ($P_n, r_n$ are from \eqref{eq:P_n})
\begin{equation}\label{eq:p_n}
  p_n=P_nr_n^{d}
\end{equation} and
\[
\mu_n(A)= \int \textbf{1}_{A \cap E_n}(x) p_n^{-1} dx
\]
where $\textbf{1}_{F}$ is the indicator function of the set $F$. Note that for every $Q\in \D_n, Q\subset E_n$ (we will denote this by $Q\in E_n$ in the following for convenience), we have $\mu_n(Q)=P_n^{-1}$. It is clear that  $\mu_n$ weakly convergence to a measure $\mu$, see \cite[Chapter 1]{Mattila1995}. We call this measure $\mu$ the natural measure on $E$.

\subsection{Results}

There exists a huge literature on computing the  `almost sure' dimensions for many other random fractal sets. We refer to \cite{Falconer1997, Falconer2003,  Fraser2014, Ga, Henna, Mauldin, Ojala} and reference therein. For the general estimations and the almost sure dimensions of these random Cantor sets, we have the following result. Let
\begin{equation}\label{eq:P(n,n+k)}
r(n,n+k)=\left(\prod_{i=n}^{n+k} M_{i}\right)^{-1},\, P(n,n+k)=\prod_{i=n}^{n+k} N_{i}.
\end{equation}

Denote
\begin{equation}\label{s_1}
s_1= \liminf_{n\rightarrow \infty} \frac{\log P_n }{-\log r_n},\, s_2=\limsup_{n\rightarrow \infty} \frac{\log P_{n+1} }{-\log r_n +\frac{1}{d}\log N_{n+1}},
\end{equation} 
and
\begin{equation}\label{eq:assouad}
s_3= \limsup_{k\rightarrow \infty} \sup_{n\in \NN} \frac{\log P(n,n+k) }{ -\log r(n,n+k)}.
\end{equation}
Furthermore let
\begin{equation}\label{eq:ts}
t^{*}= \liminf_{n\rightarrow \infty} \frac{\log P_n}{-\log r_{n+1}-\frac{1}{d}\log N_{n+1}}, \,\,
 s^{*}=\limsup_{n\rightarrow \infty}\frac{\log P_{n}}{-\log r_n}.
\end{equation}

Note that if the $N_n$ are bounded then $t^{*}=s_1$ and $s^{*}=s_2.$ Figure \ref{figure2} `explains' why there is $r_{n}/N_{n+1}^{\frac{1}{d}}$ in the definition of $s_2$. Figure \ref{figure3} `explains' why there is $r_{n+1}N_{n+1}^{\frac{1}{d}}$ in the definition of $t^{*}$. 

\begin{figure}
\includegraphics[]{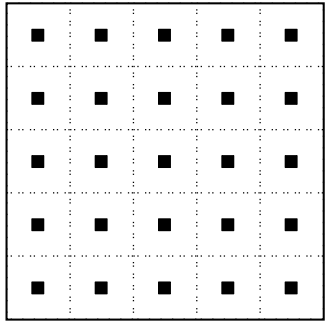}
\caption{ 
There are $N_{n+1}$  subcubes (dark cubes) of $Q$ ($Q\in E_n$) which belongs to $E_{n+1}$, and they uniformly distributed inside the cube $Q$. Thus there are nearly $N_{n+1}$ subcubes of $Q$ with side length $r_n/N_{n+1}^{\frac{1}{d}}$ (depends if $N_{n+1}^{\frac{1}{d}}$ is an integer or not) which intersect $E$.  In the end, we have $P_{n+1}$ interior disjoint cubes with side length $r_n/N_{n+1}^{\frac{1}{d}}$ which intersect $E$. This come out the definition of  $s_2$.}
\label{figure2} 
\end{figure}

\begin{figure}
\bigskip
\medskip
\resizebox{0,4\textwidth}{!}{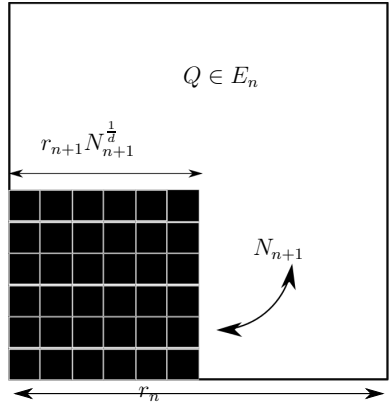}
\caption{ 
There are $N_{n+1}$ subcubes of $Q$ which belongs to $E_{n+1}$, and all of them accumulate at the left bottom of $Q$. Thus we can consider these $N_{n+1}$ subcubes as one cube with side length near $r_{n+1}N_{n+1}^{\frac{1}{d}}$ (depends if $N_{n+1}^{\frac{1}{d}}$ is an integer or not), and there are $P_n$ such cubes. This come out the definition of  $t^{*}$.}
\label{figure3} 
\end{figure}

\begin{theorem}\label{thm:dimension}
(1) For any $E\in \Omega$, we have
\[
t^{*}\leq \dim_H E \leq \ldim_B E\leq s_1.
\]

(2)  For any $E\in \Omega$, we have
\[
s^{*}\leq \dim_P E \leq \udim_B E \leq s_2.
\]

(3) The almost sure  Hausdorff dimension and lower box dimension are maximal, i,e., almost surely
\[
\dim_H E= \underline{\dim}_B E=s_1.
\]

(4) The almost sure packing dimension and upper box dimension are maximal, i,e., almost surely
\[
\dim_P E = \udim_B E =s_2.
\]

(5) For any $E\in \Omega$, we have $\dim_A E=s_3$ provided $\{N_k\}$ is bounded. Otherwise, almost surely $\dim_A E=d$.
\end{theorem}

We can also regard the space $\Omega$ as a subclass of Moran sets. The dimensional properties of Moran sets have been studied extensively, we refer to \cite{ FengWenWu1997, K, Li, Moran, PengWangWen, Wen} and reference therein. The results of Theorem \ref{thm:dimension} are similar to the dimensional results of one dimensional homogeneous Cantor sets (uniform Cantor sets). An interesting fact is that they have the `same' dimensional formulas (for our case $d=1$). For Hausdorff, lower box, upper box, and packing dimensions  of one dimensional homogeneous Cantor sets, see \cite{FengWenWu1997}. For Assouad dimension of one dimensional homogeneous Cantor sets, see \cite{PengWangWen}. The Figure \ref{figure3} corresponds to  the partial homogeneous Cantor sets of  \cite{FengWenWu1997}. 

\begin{remark}
The above statements (1) and (2) generalize the results 
of \cite{FengWenWu1997} from one dimensional Moran sets to our model, and the statement (5) when $N_k$ are bounded generalize the result of \cite{PengWangWen} from homogeneous Cantor sets to our model. The proof of $\dim_H E \geq t^{*}$ is adapted from \cite[Theorem 2.1]{FengWenWu1997} to our setting, while the method for the proof of $\udim_B E\leq s_2$ is different from that of \cite{FengWenWu1997}. The proof of the statement (5)  when $N_k$ are bounded generalize the method in \cite{PengWangWen} to high dimension. Our main contribution of Theorem \ref{thm:dimension} is to determine the almost sure dimensions of these random cantor sets for the case when $\{N_k\}_{k\in\NN}$ is unbounded. Our method combines geometric and probability estimates on the distribution of these random Cantor sets. 
\end{remark}

Recall that $(\Omega, d_H)$ is a compact metric space. For the typical dimensions of these random Cantor sets, we have the following result. For some related results we refer to \cite{FengWu, Fraser2015, FraserMiaoShasha}

\begin{theorem}\label{thm:typical}
(1) The typical Hausdorff dimension and lower box dimension are minimal, i.e., for a typical $E\in \Omega,$ we have
\[
\dim_H E= \underline{\dim}_B E= t^{*}.
\]

(2) The typical packing dimension and upper box dimension are maximal, i.e., for a typical $E\in \Omega$, we have
\[
\dim_P E = \udim_B E  =s_2.
\]

(3) If $\{N_k\}$ is unbounded, then for a typical $E\in \Omega$, we have 
\[
\dim_A E =d.
\]
\end{theorem}

Note that the typical Hausdorff dimension and  lower box dimension are as small as possible, but the almost sure Hausdorff dimension and lower box dimension are as large as possible. Furthermore the packing dimension, upper box dimension and Assouad dimension are as large as possible in the sense of both almost sure dimension and typical dimension.

For the  local dimensions of the natural measures supported on  these random Cantor sets, we have the following result. Let 
\[
s^{**}=\limsup_{n\rightarrow \infty}\frac{\log P_{n+1}}{-\log r_n}.
\]

\begin{theorem}\label{thm:localdimension}
(1) For any $E\in \Omega, x\in E$, we have 
\[  
t^{*}\leq \ldim(\mu,x)\leq s_1.
\]

(2) For any $E\in \Omega, x\in E$, we have 
\[  
s^{*}\leq \udim(\mu,x)\leq s^{**}.
\]

(3) For $\PP$-almost all $E\in \Omega$, and $\mu$ almost every $x\in E$, we have 
\[
\ldim(\mu, x)= s_1. 
\]

(4) For $\PP$-almost all $E\in \Omega$, and $\mu$ almost every $x\in E$, we have 
\[
\udim(\mu, x)= s_2. 
\]
\end{theorem}


Same kind of results have been obtained for other ``random'' measures,  we refer to \cite{FalconerMiao} and reference therein. For the local dimensions of the Moran measures, we refer to \cite{KLS, LiWu, LouWu}.

\begin{remark}\label{rem:ttt}
The dimension of a set  has essential connection with the local dimension of the measure on it, we refer to  \cite[Proposition 2.3-2.4]{Falconer1997} for more details. In fact there are some overlaps between our Theorem \ref{thm:dimension} and Theorem \ref{thm:localdimension}. Actually Theorem \ref{thm:localdimension} (3)-(4) combined with the Propostion 2.3 of \cite{Falconer1997} and Theorem \ref{thm:dimension} (1)-(2) implies Theorem \ref{thm:localdimension} (3)-(4). We present more details in the following. 

Theorem \ref{thm:localdimension} (3) and  \cite[Proposition 2.3 (a)]{Falconer1997} implies that almost surely $\dim_H E \geq s_1$. Combining this with Theorem \ref{thm:dimension} (1) which gives $\ldim_B E \leq s_1$ for any set $E\in \Omega$, we obtain Theorem \ref{thm:dimension} (3). 

Theorem \ref{thm:localdimension} (4) and  \cite[Proposition 2.3 (c)]{Falconer1997} implies that almost surely $\dim_P E \geq s_2$. Combining this with Theorem \ref{thm:dimension} (2) which gives $\udim_B E \leq s_2$ for any set $E\in \Omega$, we obtain Theorem \ref{thm:dimension} (4). 

Since our methods for Theorem \ref{thm:dimension} (3)-(4) and Theorem \ref{thm:localdimension} (3)-(4) are different, and the methods are interesting on it's own, we present them separately.
\end{remark}

For the hitting probability of these random Cantor sets, we consider the special case that $M_k=M$ and $N_k= N$ for all $k\in \NN$, and we have the following result. Note that the result is similar to the hitting probability of fractal percolation (see \cite[Theorem 9.5 ]{MPeres}) and  random covering sets (see \cite{JJK}).

\begin{theorem}\label{thm:hitting}
Let $F$ be a Borel subset  of $[0,1]^d$ with $\dim_H F =\alpha$ and $s= \log N/ \log M$. Then  we have

(1) If $\alpha < d-s$ then almost surely $E \cap F$ is empty.

(2) If $\alpha > d-s$ then $E$ intersects $F$ with positive probability. 

(3) If $\alpha > d-s$ then $\parallel\dim_H ( E  \cap F)\parallel_\infty =\alpha+s -d$, where the norm
is an essential supremum in the underlying probability space.
\end{theorem}

Note that under the condition $\sup_{k\in \NN}M_k<\infty$, it will become much easier to prove Theorem \ref{thm:dimension}, Theorem \ref{thm:typical}, and Theorem\ref{thm:localdimension}. Our main contribution of this project is to deal with the case when $\{N_k\}_{k\in \NN}$ is unbounded.

The paper is organised as follows. In Section \ref{section:lemma} we will show some lemmas for later use. Theorems \ref{thm:dimension}, \ref{thm:typical}, \ref{thm:localdimension}, and \ref{thm:hitting} are proved in Sections \ref{section:dimension}, \ref{section:typical}, \ref{section:localdimension}, and \ref{section:hitting} respectively. We conclude  with additional results and open problems in Section \ref{section:remarks}.

\section{Preliminary lemmas}\label{section:lemma}

We show some useful lemmas in this section.

\begin{lemma}\cite[Proposition 2.3]{Falconer1997}\label{lem:lp}
Let $E\subset \RR^{d}$ be a Borel set and let $\mu$ be a finite measure. If $\udim(\mu,x)\geq s$ for all $x \in E$ and $\mu(E)>0$ then $\dim_P E \geq s.$
\end{lemma}

\begin{lemma}\cite[Corollary 3.9]{Falconer2003}\label{lem:bp}
Let $E \subset \RR^d$ be a compact and such that $\udim_B ( E\cap V ) = \udim_B E$ for all open sets $V$ that intersects $E$. Then $\dim_P E= \udim_B E$.
\end{lemma}

For convenience we put an easy fact about Assouad dimension as the following lemma. For further basic facts on Assouad dimension, we refer to \cite{Luukkainen, PengWangWen}.

\begin{lemma}\label{lem:wen}
Let $E \subset \RR^d$. If there are sequences $\{R_n\}$ and $\{r_n\}$ of positive real numbers with $R_n / r_n\rightarrow \infty$ as $n\rightarrow \infty$, such that for every $n$ there exists $x\in E$ with 
\[
\mathcal{N}\left(E\cap B\left(x, R_n \right), r_n \right) \geq \left(\frac{R_n}{r_n} \right)^{s}, 
\]
then we have $dim_A E \geq s$.
\end{lemma}

The following estimate will be used in the proof of Lemma \ref{cor:intersection}.
\begin{lemma}\label{lem:erdos}
Let  $A=\{1,\cdots, k\}, B= \{1,\cdots, m\}, 1 \leq k \leq m$. Now we randomly choose $n ~ (n\leq m)$ numbers from $B$ in the same way as our construction of random Cantor sets (we randomly choose a number from $B$ uniformly, then we randomly choose an other number among the remaining $m-1$ numbers uniformly, and continue this process until we obtain $n$numbers). Let $K$ be the random chosen $n$ numbers, then 
\[
\PP(A \cap K \neq \emptyset) \geq 1- e^{-\frac{nk}{m}}.
\]
\end{lemma}
\begin{proof}
Note that the random set $K$ will intersect (hit) the set $A$ with probability one when $k+n >m.$ In the following, we assume that $k+n \leq m$. Let $K=\{x^1,\cdots, x^n\}$ where $x^i$ means the $i$-th chosen number. Let $A_i$ be the event $\{K:x^i \in B \backslash A \}$, then 
\begin{equation}
\begin{aligned}
\PP(\bigcap^n_{i=1} A_i)&=\PP(A_1) \prod^n_{i=2} \PP(A_{i} \,\big|\, \bigcap^{i-1}_{j=1}A_j)\\
& =\prod^{n-1}_{j=0} (1-\frac{k}{m-j})\\
&\leq e^{-\frac{nk}{m}}.
\end{aligned}
\end{equation}
By the fact that the event $(A \cap K \neq \emptyset )$ is the complement of the event $\bigcap^n_{i=1} A_i$, we complete the proof.
\end{proof}

The following estimate will be used in the proof of Lemma \ref{lem:box}. For more details on large deviations estimates, see \cite[Appendix A]{Alon}.
\begin{lemma}\label{lem:law of large numbers}
Let $\{X_i\}_{i=1}^n$ be a sequence nonegative independent random variables with $ X_i \leq N$ and  $\EE(X_i)\geq N/2$ for all $1\leq i \leq n$. Then 
\[ 
\PP( \sum^n_{i=1} X_i < N n/ 8)\leq e^{-n/8}.
\]
\end{lemma}
\begin{proof}
Let $\lambda =1/N$. We apply Markov's inequality to the random variable $e^{-\lambda \sum^n_{i=1} X_i}$. This gives 
\begin{equation}\label{eq:markov}
\begin{aligned}
\PP( \sum^n_{i=1} X_i < Nn/8)&=\PP (e^{-\lambda \sum^n_{i=1} X_i}  > e^{-n/8})\\
& \leq  e^{n/8} \EE(e^{-\lambda \sum^n_{i=1} X_i})\\
&=e^{n/8} \prod^n_{i=1} \EE( e^{-\lambda X_i}),
\end{aligned}
\end{equation}
the last equality holds since $\{X_i\}_i$ is a sequence independent random variables. 

For any  $t \in [0,1]$ we have  
 \[
 e^{-t} \leq 1-t/2.
 \] 
Since $\lambda X_i \in [0, 1]$ for all $1\leq i \leq n$, we have that for all $1\leq i \leq n,$
 \[
 e^{-\lambda X_i}\leq 1-\lambda X_i/2, 
 \]
 and hence 
\[
\EE( e^{-\lambda X_i}) \leq 1- \EE(\lambda X_i/2)\leq e^{-1/4}.
\]
Combining this with \eqref{eq:markov}, we finish the proof.
\end{proof}

\section{Bounds on dimensions and almost sure dimension}\label{section:dimension}

\begin{proof}[Proof Theorem \ref{thm:dimension} (1)] For any $E\in \Omega$ and $k\in \NN$, we have $\mathcal{N}(E, r_k)\leq P_k$, and hence 
\[
\ldim_B E \leq \frac{\log P_k}{-\log r_k} = s_1. 
\]

For convenience, let $\ell_k=r_{k+1}N_{k+1}^{\frac{1}{d}}, k \in \NN$.
Suppose $t^{*}>0$ ($t^{*}=0$ is the trivial case). For any $0<t<t^{*}$, by the definition of $t^{*}$, there exist $k_0$ such that for any $k\geq k_0$,
\begin{equation}\label{eq:t_1}
P_k \geq \ell_k^{-t}.
\end{equation}

Let $E \in \Omega$ and $\mu$ be the natural measure on $E$. We intend to show that $\mu (B(x,r)) \leq C r^t$ for any ball $B(x,r)$ with $r\leq r_{k_0}$ where $C$ is a constant. For $0<r\leq r_{k_0}$, there exists $k$ such that $r_{k+1}< r\leq r_k$. 

Case 1. $\ell_k\leq r\leq  r_k$. In this case, the ball $B(x,r)$ intersects at most $3^d$ cubes of $E_k$, hence 
\begin{equation}
\mu(B(x,r)) \leq 3^dP_k^{-1} \leq 3^d \ell_k^t \leq 3^dr^t
\end{equation}

Case 2. $r_{k+1}<r<\ell_k$. In this case, observe that there exists a constant $C=C(d)$ such that any ball $B(x,r)$ can intersects at most $C(\frac{r}{r_{k+1}})^d$ cubes of $E_{k+1}$, and hence 
\begin{equation}
\mu(B(x,r))\leq C\left(\frac{r}{r_{k+1}}\right)^d N_{k+1}^{-1}P_k^{-1}\leq Cr^d   \ell_k^{t-d} \leq C r^t.
\end{equation}
Thus the mass distribution principle \cite[Chapter 4]{Falconer2003} implies  that $\dim_H E \geq t$. Since this holds for any $t<t^{*}$, we obtain that $\dim_H E \geq t^{*}$. 
\end{proof}


\begin{proof}[Proof of Theorem \ref{thm:dimension} (2)] 
 For each $k\in \NN$, let $\ell_{k+1}=r_k/(N_{k+1})^{\frac{1}{d}}$. For any $\delta>0$, there exists $k$ such that $r_{k+1}<\delta\leq r_k.$ 

Case 1.  $r_{k+1}<\delta < \ell_{k+1}$.
In this case we have $\mathcal{N}(E,\delta)\leq P_{k+1}$, and hence
\begin{equation}\label{eq:boxcase1}
\frac{\log N(E,\delta)}{-\log \delta} \leq \frac{\log P_{k+1}}{-\log \delta} \leq \frac{\log P_{k+1}}{-\log     \ell_{k+1}}.
\end{equation}

Case 2. $  \ell_{k+1}\leq \delta \leq r_k$.
In this case, we have  $N(E,\delta)\leq C  P_k  \left( r_k / \delta \right)^d$.
Thus 
\begin{equation}
\begin{aligned} \label{eq:boxcase2}
\frac{\log N(E,\delta)}{-\log \delta} &\leq \frac{\log P_{k} r_k^d}{-\log \delta}+d+\frac{\log C}{-\log \delta}\\
&\leq \frac{\log P_{k} r_k^d}{-\log    \ell_{k+1}}+d+\frac{\log C}{-\log r_k}\\
& = \frac{\log P_{k+1}}{-\log    \ell_{k+1}}+\frac{\log C}{-\log r_k}.
\end{aligned}
\end{equation}
Taking the upper limit of \eqref{eq:boxcase1} and \eqref{eq:boxcase2}, we obtain that 
\[ 
\udim_B E \leq s_2.
\]

Suppose $s^{*}>0$.  For any $t<s^{*}$ there exists a sequence of numbers $\{k_j\}_{j\geq 1} \subset \NN$ with $k_j \rightarrow \infty$ as $j \rightarrow \infty$, such that $P_{k_j} \geq r_{k_j}^{-t}$ for all $j\in \NN.$
Let $x \in E,$ then we have 
\[
\mu(B(x, r_{k_j})) \leq 3^d P_{k_j}^{-1}\leq 
3^dr_{k_j}^{t},
\]
and hence $\udim(\mu,x)\geq t$. Since this holds for all $x\in E$, together with Lemma \ref{lem:lp}  we have that $\dim_P E \geq t$. By the arbitrary choice of $t<s^{*}$, we obtain that $\dim_PE \geq s^{*}$. Thus we complete the proof. 
\end{proof}

\subsection{Almost sure Hausdorff and lower box dimensions}

Let $\partial Q$ be the boundary of $Q$, define 
\[
\widetilde{B}_n=\bigcup_{Q \in \D_n} \partial Q \text{ and } \widetilde{B}=\bigcup_{n\in \NN} \widetilde{B}_n.
\]
It is clear that $\widetilde{B}$ has zero Lebesgue measure and for any $x \in [0,1]^d\backslash \widetilde{B}$, 
\[
\PP(x \in E_n)=p_n.
\] 
Recall that $p_n=P_nr_n^{d}$. For the purpose of estimating the lower bound for Hausdorff dimension, we need the following estimate. 

\begin{lemma}\label{lem:twopoint}
For any $ \varepsilon > 0$ there exists positive constant $C=C(\varepsilon, d)$, such that 
\begin{equation}\label{eq:twopoint}
\frac{\PP( x \in E_{n}, y \in E_{n})}{p_{n}^{2}} \leq C d(x,y)^{s_1-d-\varepsilon}
\end{equation}
for all $n \in \NN$ and $x, y \in [0,1]^d$.
\end{lemma}
\begin{proof}
For any $ \varepsilon > 0$, by the definition of $s_1$, there exists $N \in \NN$, such that  $P_n\geq r_n^{-s_1+\varepsilon}$ for all $ n\geq N$ which implies that
\begin{equation}\label{eq:pp_n}
p_n=P_nr_n^{d}\geq r_n^{d-s_1+\varepsilon}.
\end{equation} 
We first assume that $x, y \in [0,1]^d\backslash \widetilde{B}$ (to make sure that for any $n\in \NN$ there exists unique $Q, Q' \in \D_n$ with $x\in Q$ and $y\in Q'$). There is  $k\in \NN$ such that 
\[
\sqrt{d}r_{k+1}<d(x,y)\leq \sqrt{d}r_k. 
\]
It follows that there exists two distinct cubes $Q_x, Q_y \in \D_{k+1}$ such that $x\in Q_x$ and $y\in Q_y$. Therefore, for any $n> k$ we have 
\begin{equation}
\begin{aligned}\label{eq:nnnn}
\mathbb{P}( x \in E_{n},  y \in E_{n}) &=  \mathbb{P}( y \in E_{n}\mid x \in E_{n})\mathbb{P}( x \in E_{n}) \\
&\leq p_n\prod^{n}_{i=k+1} N_iM_i^{-d}=p_n^{2}p_k^{-1}.
\end{aligned}
\end{equation}

Now we turn to the estimate of \ref{eq:twopoint}. Case 1. $n\leq N$. In this case we have
\[ 
\frac{ \mathbb{P}( x \in E_{n},   y \in E_{n})}{p_{n}^{2}} \leq p_n^{-2}\leq  p_N^{-2}\leq  p_N^{-2} (\sqrt{d})^{d-s_1+\varepsilon} d(x,y)^{s_1-d-\varepsilon}.
\]

Case 2. $n > N$. In this case there will appear three  subcases  depending on $d(x,y)$.

Subcase 1. $ d(x,y)\leq  \sqrt{d}r_{n}$. In this case, 
\[
\mathbb{P}( x \in E_{n},  y \in E_{n}) =  \mathbb{P}( y \in E_{n} \, |\, x \in E_{n})\mathbb{P}( x \in E_{n}) 
\leq p_n.
\]
Combining this with the estimate \eqref{eq:pp_n} we obtain
\begin{equation}
\begin{aligned}
\frac{ \mathbb{P}( x \in E_{n}, y  \in E_{n})}{p_{n}^{2}}& \leq  p_n^{-1} \leq r_n^{s_1-d-\varepsilon}\\
&\leq \sqrt{d}^{d-s_1+\varepsilon} d(x,y)^{s_1-d-\varepsilon}.
\end{aligned}
\end{equation}

Subcase 2. $d(x,y) > \sqrt{d}r_N$. Applying the estimate \eqref{eq:nnnn} we have
\begin{equation}
\begin{aligned}
 \frac{ \mathbb{P}( x \in E_{n},   y \in E_{n})}{p_{n}^{2}} &\leq p_k^{-1} \leq p_N^{-1}\\
 &\leq  p_N^{-1} (\sqrt{d})^{d-s_1+\varepsilon} d(x,y)^{s_1-d-\varepsilon}.
\end{aligned}
\end{equation}

Subcase 3.  $\sqrt{d}r_{n}< d(x,y)\leq \sqrt{d}r_{N}$. Applying the estimates \eqref{eq:pp_n} and \eqref{eq:nnnn}, we have 
\begin{equation}
\begin{aligned}
 \frac{ \mathbb{P}( x \in E_{n},   y \in E_{n})}{p_{n}^{2}} &\leq p_k^{-1} \leq r_k^{s_1-d-\varepsilon}\\
 &\leq \sqrt{d}^{d-s_1+\varepsilon}
 d(x,y)^{s_1-d-\varepsilon}.
\end{aligned}
\end{equation}

Let $C=p_N^{-2}\sqrt{d}^{d}$, then the estimate \eqref{eq:twopoint} holds for all $x, y \in [0,1]^d\backslash \widetilde{B}$. Note that for every point $x$ and $n\in \NN$ there exist at most $2^{d}$ cubes of $\D_n$ such that each of these cube contains $x$. It follows that for any $x, y \in [0,1]^d$
\[
\PP(x\in E_n, y\in E_n ) \leq  4^{d} \PP(Q_x  \in E_n, Q_y \in  E_n)
\]
where $Q_x , Q_y$ are two cubes of $\D_n$ which contain $x $ and $y$ separately. Thus there exists a larger constant such that the estimate \eqref{eq:twopoint} holds for all $x, y \in [0,1]^d$. For the convenience, we denote this larger constant also by $C$.
\end{proof}

\begin{proof}[Proof of Theorem \ref{thm:dimension} (3)]
By Theorem \ref{thm:dimension} (1), it is sufficient 
to prove that almost surely $\dim_H E \geq s_1$. For any $\varepsilon >0,$ there exists a positive constant $C$, such that Lemma \ref{lem:twopoint} holds.  Applying Lemma \ref{lem:twopoint}, Fatou's lemma, and Fubini's theorem, we obtain
\begin{equation}
\begin{aligned}\label{eq:ooo}
\EE  & \left(  \int  \int d(x,y)^{-s_1+2\varepsilon} d\mu(x) d\mu(y) \right) \\
&\leq\liminf_{k \rightarrow \infty}
\EE \left( \int \int d(x,y)^{-s_1+2\varepsilon}d\mu_k(x)d\mu_k(y) \right) \\
&=\liminf_{k \rightarrow \infty} \EE \left( \int \int d(x,y)^{-s_1+2\varepsilon}p_k^{-2}\textbf{1}_{E_k\times E_k}(x,y)dxdy \right) \\
&\leq C \int_{[0,1]^d} \int_{[0,1]^d} d(x,y)^{-s_1+2\varepsilon}d(x,y)^{s_1-d-\varepsilon}dxdy < \infty.
\end{aligned}
\end{equation}
This implies that a.s. 
\begin{equation}\label{eq:ttt}
\int  \int d(x,y)^{-s_1+2\varepsilon} d\mu(x) d\mu(y) <\infty.
\end{equation}
Thus by applying the energy argument \cite[Theorem 4.13]{Falconer2003}, we have that  almost surely $\dim_H E\geq s_1-2\varepsilon$. By the arbitrary choice of $\varepsilon$, we obtain that almost surely $\dim_H E \geq s_1.$ 
\end{proof}

\begin{remark}
Note that the estimate \eqref{eq:ttt} implies that $\ldim(\mu,x)\geq s_1-2\varepsilon$ for $\mu$ almost all $x\in E$, for a proof see the argument in \cite[Theorem 4.13]{Falconer2003}. Together with the estimate \eqref{eq:ooo} and the arbitrary choice of $\varepsilon$ we obtain that for $\PP$-almost all $E\in \Omega$, and $\mu$ almost every $x\in E$, we have $\ldim(\mu, x)\geq s_1$. As we claimed before in Remark \ref{rem:ttt}, we will present a different proof in Section \ref{section:localdimension}.
\end{remark}

\subsection{Almost sure packing and upper box dimensions} \label{ll}

For every $Q \in \D_k, k\in \NN$, we define the random set
\[
E_{k+1}(Q)=\{Q': Q' \subset Q, Q' \in E_{k+1}\}.
\]
Recall that $Q'\in E_{k+1}$ means that $Q' \in \D_{k+1}$ and $Q'\subset E_{k+1}$. In the following we are going to show that the set $E_{k+1}(Q)$ is fairly uniformly distributed (this motivated the formula of the upper box dimension).

Let $N^*_{k+1}= \lfloor N_{k+1}^{\frac{1}{d}} \rfloor ^d$ where $\lfloor x \rfloor$ denotes the integer part of $x$. For every $Q \in \D_k$, we divide it into $N^*_{k+1}$ interior disjoint closed subcubes with side length 
\begin{equation}\label{eq:delta'}
r^{*}_{k+1}=r_k/(N^*_{k+1})^\frac{1}{d},
\end{equation} 
and denote by $\C(Q, N^*_{k+1})$ the collection of these subcubes. For every $\widetilde{Q}\in \C( Q,N^*_{k+1})$, define
\[
I(Q, \widetilde{Q}, \D_{k+1})=\{ Q' \in \D_{k+1}: Q' \subset Q \text{ and } Q' \cap \widetilde{Q} \neq \emptyset\}.
\]
By a volume argument, we have 
\begin{equation}\label{eq:number}
\#I(Q, \widetilde{Q}, \D_{k+1})  \geq \frac{M_{k+1}^d}{N^*_{k+1}} 
\end{equation}
where $\# J$ denotes the cardinality of a set $J$.  and random variable
\begin{equation}\label{eq:randomdistributied}
X_Q =\# \{ \widetilde{Q} \in \C(Q, N^*_{k+1}) : \widetilde{Q} \cap E_{k+1}(Q) \neq \emptyset\}.
\end{equation}

Figure \ref{figure4} shows the relative position of the above geometric objects. 

\begin{figure}

\resizebox{0,4\textwidth}{!}{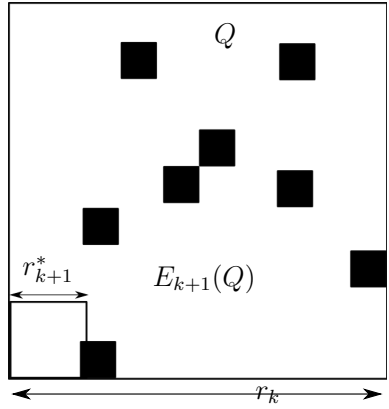}
\caption{ A cube $Q\in E_k$, the set $E_{k+1}(Q)$ consisting of the dark cubes, a cube $\widetilde{Q}\in \mathcal{C}(Q, N_{k+1}^*)$. }
\label{figure4} 
\end{figure}

\begin{lemma}\label{cor:intersection}
Let $Q \in \D_k$, then for every $\widetilde{Q} \in \C(Q, N^*_{k+1})$ we have 
\begin{equation}
\PP\left( \widetilde{Q} \cap  E_{k+1}(Q) \neq \emptyset  \, \Big|\, Q \in E_k \right)\geq 1/2, 
\end{equation}
and hence  
\begin{equation}
\EE \left( X_Q \,\Big|\, Q \in E_k \right)\geq N^*_{k+1}/2. 
\end{equation}
\end{lemma}
\begin{proof}
Applying  Lemma \ref{lem:erdos} for
\[
A(Q)= \{Q'\in \D_{k+1}: Q' \subset Q , Q' \cap \widetilde{Q} \neq \emptyset\},
\]
\[
B(Q)=\{Q'\in \D_{k+1}: Q'\subset Q\}),
\]
and the estimate \eqref{eq:number}, we obtain
\begin{equation}
\begin{aligned}
\PP ( \widetilde{Q} &\,\cap \, E_{k+1}(Q) \neq \emptyset  \, \big|\, Q \in E_k )\\ 
&\geq 1-\exp\left (-\frac{N_{k+1}}{M_{k+1}^d} \# I(Q, \widetilde{Q}, \D_{k+1}) \right)\\
&\geq 1-\exp \left(-\frac{N_{k+1}}{N^*_{k+1}}\right)\\
&\geq 1-e^{-1}\geq 1/2.
\end{aligned}
\end{equation}
It follows that 
\begin{equation}
\begin{aligned}
\EE(X_Q \, \big|\, Q \in E_k)
&= \sum_{\widetilde{Q} \in  \C(Q, N^*_{k+1}) }\PP \left(  \widetilde{Q} \cap  E_{k+1}(Q) \neq \emptyset  \, \Big|\, Q \in E_k \right)\\
& \geq N^*_{k+1}/2.
\end{aligned}
\end{equation}
Thus we complete the proof.
\end{proof}

The following proposition contains the second statement of Theorem \ref{thm:dimension} $(4)$.

\begin{proposition}\label{lem:box}
Almost surely $\udim_B E=s_2.$
\end{proposition}
\begin{proof}
If $\{N_k\}$ is bounded then Theorem \ref{thm:dimension} (2) implies that $\udim_B E= s_2$ for all $E\in \Omega$. Furthermore \ref{thm:dimension} (2) clams that $\udim_B E \leq s_2$ for any $E\in \Omega$. Thus it is sufficient to prove that almost surely $\udim_B E\geq s_2$ for the case that $\{N_k\}$ is unbounded.

Suppose $s_2>0$. By the definition of $s_2$ (see \eqref{s_1}), for any $0<\varepsilon <s_2 $ there exist a sequence  $\{n_k\}_{k\in \NN}\subset \NN, n_1\leq n_2\leq \cdots$ such that 
\begin{equation}\label{eq:mmmm}
P_{n_{k}+1}\geq  (r_{n_k}/(N_{n_k+1})^{\frac{1}{d}})^{-s_2+\varepsilon}.
\end{equation}
Observe that if $\{N_{n_k+1}\}_{k\in \NN}$ is bounded, then $\udim_B E \geq s_2$ for any $E \in \Omega$. Thus  
we suppose that  $N_{n_k+1} \nearrow \infty$ as $n_k \rightarrow\infty$, and $N_{n_1} \geq 2^d$. It follows that for all $k \in \NN$,
\begin{equation}
N^*_{n_k+1}\geq 2^{-d}N_{n_k+1}.
\end{equation}

For each  $k \in \NN, Q \in E_{n_k}$, by Corollary \ref{cor:intersection} we have 
\[
\EE (X_Q) \geq N^*_{n_k+1}/2.
\]
Furthermore, conditional on $E_{n_k}$, we have that $X_Q$ and $X_{Q'}$ are independent for any two distinct cubes $Q, Q' \in E_{n_k}$.  Thus applying Lemma \ref{lem:law of large numbers}, we obtain that
\begin{equation}\label{eq:applylawof}
\PP \left( \sum_{Q \in E_{n_k}} X_Q <  N^*_{n_k+1} P_{n_k}/8 \,\Big|\, E_{n_k} \right) \leq e^{-P_{n_k}/8}.
\end{equation}

Recall that 
\[
r^{*}_{k+1}=r_k/(N^*_{k+1})^\frac{1}{d}.
\] 
By  elementary geometry, there exists a positive constant $C=C(d)$ such that 
\[
\mathcal{N}(E, r^{*}_{n_k+1})\geq C\sum_{Q \in \D_{n_k}} X_Q.
\]
For each $k\in \NN$, define the event
\[
A_k = \left( \mathcal{N}(E, r^{*}_{n_k+1})< C N^*_{n_k+1}P_{n_k}/8\right).
\]
Combining this with the estimate \eqref{eq:applylawof}, we have  
\begin{equation}
\begin{aligned} 
\PP(A_k \big|E_{n_k} ) &\leq \PP (\sum_{Q \in E_{n_k}} X_Q< N^*_{n_k+1} P_{n_k}/8 |E_{n_k})\\
& \leq e^{- P_{n_k}/8}.
\end{aligned}
\end{equation}
It follows that 
\[
\PP(A_k )\leq e^{- P_{n_k}/8}.
\]
Since $N_{n_k}\nearrow \infty$ as $k \rightarrow \infty$, thus there exists $k_0\in \NN$ such that $P_{n_k}=\prod_{i=1}^{k} N_{n_i} \geq n_k$ for all  $k\geq k_0$, and hence
\[
\sum^{\infty}_{k=1} e^{- P_{n_k}/8} < \infty.
\]
Applying the Borel-Cantelli lemma, we obtain
\[
\PP(\bigcup^\infty_{m=1} \bigcap^\infty_{k = m}A_k^c)=1,
\]
where $A_k^c$ means the complement of $A_k$. Thus we obtain that for almost any $\omega \in \Omega$, there exists $k_\omega$, such that $\omega\in A_k$ for all $k\geq k_\omega.$ It follows that for every $k \geq k_\omega$, we have 
\[
\mathcal{N}(E^\omega, r^{*}_{n_k+1}) \geq  
CN^*_{n_k+1}P_{n_k}/8.
\]
Since $N^*_{n_k+1}\geq 2^{-d}N_{n_k+1}$ and $r^{*}_{n_k+1}\geq r_{n_k}/(N_{n_k+1})^{\frac{1}{d}}$, we have
\[
\frac{\log \mathcal{N}(E^\omega, r^{*}_{n_k+1})}{-\log r^{*}_{n_k+1} } \geq \frac{\log C N_{n_k+1}^{*}P_{n_k}/8}{-\log r_{n_k}/(N_{n_k+1})^{\frac{1}{d}} }  
\]
holds for all $k \geq k_\omega$, and hence by the estimate \eqref{eq:mmmm} we have $\udim_B E^\omega \geq s_2-\varepsilon$. By the arbitrary choice of $\varepsilon$ we finish the proof.
\end{proof}

Now we intend to show that almost surely $\dim_P E=s_2$.

\begin{proof}[Proof of theorem \ref{thm:dimension} (4)]
 

Recall that  $E =\bigcap_{n\in \NN} E_n$.  Let $\{x_n\}_{n\geq 1}$ be a dense subset of  $[0,1]^d$, and $B_n:=B(x_n, 1/n)$ be an open ball. For every $n\in \NN$, by the homogeneous structure of our random Cantor sets, we obtain that almost surely on $E \cap B_n \neq \emptyset,$ 
\[
\udim_B (E \cap B_n)= s_2.
\]
It follows that almost surely for any $B_n$ (here the order of `almost surely' and `for every $n\in \NN$' is different from above) with $E \cap B_n\neq \emptyset$, 
\[
\udim_B (E\cap B_n) =s_2.
\]
Observe that for any $E\in \Omega$ and any open set $U$ with $E\cap U \neq \emptyset$, there there is a ball $B_n$ for some $n\in \NN$ such that 
\[
B_n \subset U, \, B_n \cap E \neq \emptyset. 
\]
Hence for almost all $E\in \Omega$ and any open set $U \cap E \neq \emptyset,$  
\[
\udim_B (E \cap U) \geq s_2.
\]
Applying Lemma \ref{lem:bp}, we obtain that almost surely $\dim_P E\geq s_2.$ By the fact that  $\dim_P E\leq \udim_B E$ and the Proposition \ref{lem:box}, we complete the proof.
\end{proof}


\subsection{Almost sure Assouad dimension}

\begin{proof}[Proof of theorem \ref{thm:dimension} (5)]  Assume first that $\{N_k\}$ is bounded. Let $N=\sup_{k\geq 1} N_k$. By the definition of $s_3$ we have that for any $\varepsilon>0$ there exists $k_0$ such that for any $k\geq k_0$,
\begin{equation}\label{eq:sup}
\sup_{n}\frac{\log P(n,n+k)}{-\log r(n,n+k)}< s_3+\varepsilon.
\end{equation}

Let $E\in \Omega$, for any $0<r < R \leq \sqrt{d}$, there exist $n,k$, such that 
\begin{equation}\label{eq:Rr}
r_{n+1}<R\leq r_n, \,\, r_{n+k+1}<r\leq r_{n+k}.
\end{equation}

Case 1.  $k< k_0$. For any $x\in E$ we have 
\[
\mathcal{N}(B(x,R)\cap E, r) \leq 3^dN^{k_0+1}.
\]

Case 2. $k \geq k_0$. For any $x\in E$, by estimates \eqref{eq:sup} and \eqref{eq:Rr} we obtain
\begin{equation}
\begin{aligned}
\mathcal{N}(B(x,R)\cap E, r) &\leq 3^d \prod^{n+k+1}_{i=n+1}N_i \\
&\leq 3^d N_{n+1}N_{n+k+1} \left(\frac{r_{n+1}}{r_{n+k}} \right)^{s_3+\varepsilon}\\
& \leq 3^d N^2\left( \frac{R}{r} \right)^{s_3+\varepsilon}.
\end{aligned}
\end{equation}
Thus we have $\dim_A E \leq s_3+\varepsilon$. By the arbitrary choice of $\varepsilon$ we obtain $\dim_A E \leq s_3$.

For the lower bound. For any $\varepsilon >0$, there exists $k_i \nearrow \infty$ as $i \rightarrow \infty$, such that for every $i$
\begin{equation}
\sup_{n}\frac{\log P(n,n+k_i)}{-\log r(n,n+k_i)}> s_3-\varepsilon,
\end{equation}
and so there exists $n_i$ such that 
\begin{equation}\label{eq:lower}
\frac{\log P(n_i,n_i+k_i)}{-\log r(n_i,n_i+k_i)}> s_3-\varepsilon.
\end{equation}
It follows that
\begin{equation}
\begin{aligned}
\mathcal{N}(B(x,  r_{n_i})  \cap E, r_{n_i+k_i}) &\geq C P(n_i,n_i+k_i) \\
&\geq C \left(\frac{r_{n_i}}{r_{n_i+k_i}} \right)^{s_3-\varepsilon}
\end{aligned}
\end{equation}
where $C=C(d)$ is a positive constant. Applying Lemma \ref{lem:wen} and the estimate
\[
\frac{r_{n_i}}{r_{n_i+k_i}}\geq 2^{k_i}\rightarrow \infty \text{ as }i \rightarrow \infty,
\]
we obtain that $\dim_A E \geq s_3-\varepsilon$. By the arbitrary choice of $\varepsilon$ we have $\dim_A E \geq s_3$. Thus we complete the proof in the case 
when $\{N_k\}$ is bounded. 

Now suppose $\{N_k\}$ is unbounded. Since $\dim_A E \leq d$ holds for any $ E \subset [0,1]^d$, it is sufficient to show that almost surely $\dim_A E\geq d$. Let
\[
\{n_k\}_{k\in\NN} \subset \NN \text{ with } N_{n_k+1}\nearrow \infty \text{ as } k\rightarrow\infty.
\]
For every $Q \in \D_{n_k}$, define the event  
\[
 A= (X_Q >N_{n_k+1}^*/4 ).
 \]
Recall  the random variable $X_Q$  defined in \eqref{eq:randomdistributied}. Thus
\begin{equation}
\begin{aligned}
&\EE \left(X_Q \,\Big|\, Q \in  E_{n_k} \right)\\ 
&=\EE\left(X_Q \textbf{1}_{A} \,\Big|\, Q  \in E_{n_k} \right) +\EE\left(X_Q \textbf{1}_{\Omega \backslash A}\,\Big|\, Q \in E_{n_k} \right) \\
& \leq N_{n_k+1}^* \PP\left(A \,\Big|\, Q \in E_k\right)+N_{n_k+1}^*/4.
\end{aligned}
\end{equation}
Combining this with Corollary \ref{cor:intersection}, we have 
\begin{equation}\label{eq:intersectionnumber}
\PP\left(A \,\Big|\, Q \in E_{n_k}\right) > 1/4.
\end{equation}
For every $k\in \NN$, define the event
\[
A_k = \left( \text{ there exists } Q \in E_{n_k} \text{ such that } X_Q > N_{n_k+1}^* /4 \right).
\]
Conditional on $E_{n_k}$, recall that the cubes form $E_{n_k+1}$ are chosen independently inside 
each cube of $E_{n_k}$. Thus the random variables $X_Q$ and $X_{Q'}$ are independent for any two distinct cubes $Q, Q'$ of $E_{n_k}$. Together with the estimate \eqref{eq:intersectionnumber} we have
\[
\PP\left(A_k\,\Big|\, E_{n_k} \right) \geq 1-\left(\frac{3}{4} \right)^{P_{n_k}}.
\]
It follows that for every $k\in \NN$,
\[
\PP(A_k)\geq 1-\left(\frac{3}{4} \right)^{P_{n_k}}
\]  
Thus for any $m\geq 1,$ we have $\PP(\cup^\infty_{k=m}A_k)=1$, and hence 
\begin{equation}\label{eq:full}
\PP\left(\bigcap^{\infty}_{m=1}
\bigcup^\infty_{k=m}A_k \right)=1.
\end{equation}
It follows that  for almost all $\omega \in \Omega$, there exists $k_j=k_j(\omega)\nearrow \infty$, such that $\omega\in A_{k_j}$ for all $j\in \NN.$ Combining this with Lemma \ref{lem:wen}, we obtain that almost surely $\dim_A E \geq d$. Thus we complete the proof. 
\end{proof}

\section{Typical dimensions}\label{section:typical}

For each cube $Q$, let $z_Q\in Q$ be the nearest point of $Q$ to zero vector. For each $n\in \NN$ let 
\[
\E_n=\{\text{all the possible } E_n\}.
\]

\begin{proof}[Proof of Theorem \ref{thm:typical} (1).]
Theorem \ref{thm:dimension} (1) claims that any element of $E\in \Omega$ has $\dim_H E\geq t^{*}$. In the following we intend to show that 
a typical set $E\in \Omega$ has $\ldim_B E \leq t^{*}$.

For each $n\in \NN$, let $\varepsilon_n=2\sqrt{d}
r_{n+1}N_{n+1}^{1/d}$. For each $E_n \in \E_n$, we choose an object  $\gamma=\gamma(E_n)\in \Omega$ with 
\[
\gamma \subset E_n, \text{ and } \gamma \subset \bigcup_{Q\in E_n} B(z_Q, \varepsilon_n ).
\]
Let $\Gamma_n$ be the collection of these $\gamma(E_n), E_n \in \E_n$. Observe that for any infinite set $A\subset \NN$, the set 
\[
\{\g: \g \in \Gamma_n, n \in A\}
\]
is a countable dense subset of $\Omega$.

By the definition of $t^{*}$ there is a subsequence  $I=\{n_k\}_{k\in \NN} \subset \NN$ with 
$n_k\nearrow \infty$ as $k \rightarrow \infty$ such that 
\begin{equation}\label{eq:subsequence}
t^{*}=\lim_{k\rightarrow \infty} \frac{\log P_{n_k}}{- \log r_{n_k+1}N_{n_k+1}^{1/d}}.
\end{equation}
Let $I_m=\{n_k \in I: n_k \geq m\}$ and
\[
\G=\bigcap^{\infty}_{m=1} \bigcup_{n_k \in I_m} \bigcup_{\g \in \Gamma_{n_k}} U_{d_H}(\g, r_{n_k+1}\sqrt{d}),
\]
where $U_{d_H}(\g, \ell)$ is an open set of $(\Omega, d_H)$ with center $\g$ and radius $\ell$. Since $\{\g: \g \in \Gamma_{n_k}, k\in \NN\}$ is a countable dense subset in $\Omega,$  the set 
\[
\bigcup_{n_k \in I_m} \bigcup_{\g \in \Gamma_{n_k}} U_{d_H}(\g, r_{n_k+1}\sqrt{d}),
\]
is a dense open set in $\Omega$. It follows that the complement of $\G$ is of first category. 

Let $E\in \G$, then there is subsequence $\{q_k\}_{k\in \NN} \subset \{n_k\}_{k\in \NN}$ with $q_k\nearrow \infty$ as $k \rightarrow \infty$ and $\g_{q_k} \in \Gamma_{q_k}$ such that 
\[
E\in \bigcap^{\infty}_{k=1} U_{d_H}(\gamma _{q_k},r_{q_k+1}\sqrt{d}).
\] 
Observe that  
\[
\mathcal{N} (E, 2\varepsilon_{q_k})\leq P_{q_k}.
\]
Combining this with the definition of  $\varepsilon_{q_k}$ and the formula \eqref{eq:subsequence}, we obtain
\[
\ldim_B E \leq \liminf_{k\rightarrow \infty} \frac{\log P_{q_k}}{- \log 2\varepsilon_{q_k}}=t^{*}.
\]
Thus we complete the proof. 
\end{proof}

\begin{remark}\label{re:lowlocal}
From the construction of $\gamma _{q_k}$, it follows that for any set $E\in U_{d_H}(\gamma _{q_k},r_{q_k+1}\sqrt{d})$, there exists a constant $C>0$, such that for any $x\in E, k\in \NN$,
\[
\mu(B(x,2\varepsilon_{n_k}))\geq C P_{n_k}^{-1},
\]
and hence $\ldim(\mu, x)\leq t^{*}$.
\end{remark}

\begin{proof}[Proof of Theorem \ref{thm:typical} (2)]
Theorem \ref{thm:dimension} (2) claims that any element $E\in \Omega$ has $\udim_BE \leq s_2$. In the following we intend to show that 
a typical set $E\in \Omega$ has $\dim_P E \geq s_2$.

For each $n\in \NN$, recall that  $r_{n+1}^{*}=
r_{n}/(N_{n+1}^{*})^{1/d}$. For each $E_n \in \E_n$ we intend to choose a set $\gamma=\gamma(E_n)$ depending on the relative size of $r_{n+1}^{*}$ and $r_{n+1}$. 

Case 1. $r_{n+1}^{*}<100\sqrt{d}r_{n+1}$. In this case for each $E_n$ we choose a set $\gamma$ with 
\[
\gamma \in \Omega, \, \gamma \subset E_n, \, \gamma \cap Q \neq \emptyset \text{ for any } Q \in E_n .
\]

Case 2. $r_{n+1}^{*}\geq 100\sqrt{d}r_{n+1}$. For each $E_n$ we choose a set $\gamma\subset E_n$ with 
\[
\gamma \in \Omega, \,\, \gamma\subset \bigcup_{Q\in \C(Q, N^*_{k+1})}B(z_Q, 5\sqrt{d}r_{n+1}).
\]
The notation $\C(Q, N^*_{k+1})$ is given at the beginning of Subsection \ref{ll}. In this case, we may think $\gamma$ as those $E_{n+1}$ which the cubes of $E_{n+1}$ is well separated.

Let $\Gamma_n$ be the collection of these $\gamma(E_n).$ Observe that for any infinite set $A\subset \NN$, the set 
\[
\{\g: \g \in \Gamma_n, n \in A\}
\]
is a countable dense subset of $\Omega$. 

By the definition of $s_2$ there is a subsequence  $I=\{n_k\}_{k\in \NN} \subset \NN$ with 
$n_k\nearrow \infty$ as $k \rightarrow \infty$ such that 
\begin{equation}\label{eq:subsequence2}
s_2=\lim_{k\rightarrow \infty} \frac{\log P_{n_k+1}}{- \log r_{n_k}+ \frac{1}{d} \log N_{n_k+1}}.
\end{equation}
Let $I_m=\{n_k \in I: n_k \geq m\}$ and
\begin{equation}\label{eq:typicalassouad}
\G=\bigcap^{\infty}_{m=1} \bigcup_{n_k \in I_m} \bigcup_{\g \in \Gamma_{n_k}} U_{d_H}(\g, r_{n_k+1}\sqrt{d}).
\end{equation}
Applying the same argument as in the proof of  Theorem \ref{thm:typical} (1), we obtain that the complement of $\G$ is of first category. 

Let $E\in \G$, then there is subsequence $\{q_k\}_{k\in \NN} \subset \{n_k\}_{k\in \NN}$ with $q_k\nearrow \infty$ as $k \rightarrow \infty$ and $\g_{q_k} \in \Gamma_{q_k}$ such that 
\[
E\in \bigcap^{\infty}_{k=1} U_{d_H}(\gamma _{q_k}, r_{q_k+1}\sqrt{d}).
\] 
Let $\mu$ be the natural measure on $E$. We are going to present that there exists a positive constant $C=C(d)$ such that for any $x\in E \in U_{d_H}(\gamma _{q_k}, r_{q_k+1}\sqrt{d})$, 
\begin{equation}\label{eq:to}
\mu(B(x, r_{q_k+1}^{*}/10))\leq C P_{q_k+1}^{-1}.
\end{equation}

For the above Case 1, we have 
\[
\mu(B(x, r_{q_k+1}^{*}/10))\leq \mu(B(x, 10 \sqrt{d}r_{q_k+1})) \leq C(d)P_{q_k+1}^{-1}.
\]
Now we turn to the Case 2. Since for any $x\in E \in U_{d_H}(\gamma _{q_k}, r_{q_k+1}\sqrt{d})$ there exists at most one cube of $E_{q_k+1}$ intersects $B(x, r_{q_k+1}^{*}/10)$, we have
\[
\mu(B(x, r_{q_k+1}^{*}/10))  \leq P_{q_k+1}^{-1}.
\]
Thus we obtain the estimate \eqref{eq:to}. Together with  the formula \eqref{eq:subsequence2}, we have
\[
\udim (\mu,x) \geq  \limsup_{k\rightarrow \infty} \frac{\log C^{-1}P_{q_k+1}}{- \log r_{q_k+1}^{*}/10}\geq s_2.
\]
Since this holds for any $E\in \G$ and $x\in E,$ by Lemma \ref{lem:lp} we obtain that any $E\in \G$ has $\dim_P E \geq s_2$. Thus we complete the proof.
\end{proof}

Note that the above proof also implies that a typical $E\in \Omega$ has full Assouad dimension. We show an outline for the proof.

\begin{proof}[Proof of Theorem \ref{thm:typical} (3)]

Assume $\{n_k\}\subset \NN$ with $N_{n_k+1}\nearrow \infty$ as $k \rightarrow \infty$. Let $G$ be the set in \eqref{eq:typicalassouad}. Then the structure of $\gamma \in \Gamma_{n_k}$ and Lemma \ref{lem:wen} imply that any element of $G$ has full Assouad dimension. Thus we complete the proof.
\end{proof}

\section{Local dimensions of natural measures}\label{section:localdimension}

\begin{proof}[Proof of Theorem \ref{thm:localdimension} (1)] 
For every $x \in E$ and $k\in \NN$, we have 
\[
\mu(B(x, \sqrt{d}r_k))\geq P_k^{-1},
\] 
and hence
\[
\ldim(\mu, x)\leq \liminf_{k \rightarrow \infty} \frac{\log \mu(B(x,\sqrt{d}r_k))}{\log \sqrt{d}r_k} \leq s_1.
\]

On the other hand, it follows immediately from the proof of Theorem \ref{thm:dimension} (1) that 
\[
\ldim(\mu,x) \geq t^{*} \text{ for all }  x\in E, \,\, E\in \Omega.
\]
Thus we complete the proof.
\end{proof}

\begin{proof}[Proof of Theorem \ref{thm:localdimension} (2)] For any $x\in E, 0<r<1$, there exists $k$ such that $\sqrt{d}r_{k+1}< r\leq \sqrt{d}r_k$. Observe that
\[
\mu(B(x,r)) \geq P_{k+1}^{-1}, 
\]
and
\[
\frac{\log \mu(B(x,r))}{\log r} \leq \frac{\log P_{k+1}}{-\log \sqrt{d}r_k}.
\]
Therefore 
\[
\udim(\mu, x) \leq \limsup_{k \rightarrow \infty}\frac{\log P_{k+1}}{-\log r_k}=s^{**}.
\]

On the other hand, for any $k\in \NN$,
\[
\mu(B(x, r_k)) \leq 3^{d}P_k^{-1},
\] 
and hence
\[
\udim(\mu, x)\geq \limsup_{k \rightarrow \infty} \frac{\log P_k}{-\log r_k}=s^{*}.
\]
Thus we complete the proof.
\end{proof}

\subsection{Almost sure lower local dimension}
We start from the following Lemma.

\begin{lemma}\label{lem:massdistributionr}
For any $0< s < s_1$, there exists a positive constant $C$ such that  for any fixed $x \in [0,1]^{d}$,
\begin{equation}\label{eq:expectmeasure}
 \mathbb{E} \left( \mu_{n}\left(B(x,r) \right) \big\? x \in E_{n} \right) \leq C r ^{s},  0<r<1, n\in \NN. 
\end{equation}
Furthermore we have 
\begin{equation}\label{eq:expectmeasurefuthermore}
\EE \left(\int \mu \left(B(x,r) \right)d\mu(x) \right)\leq Cr^s.
\end{equation}
\end{lemma}
\begin{proof}
For $0<s<s_1$, by the definition of $s_1$, there exists $N$ such that for all $n \geq N$, $P_n \geq r_n^{-s}$. For $0<r<1$, there exists $k$ such that $r_{k+1}\leq r <r_k$. 

Case 1.  $n <N$. In this case we have 
\begin{equation*}
\begin{aligned}
\EE&(\mu_n(B(x,r) ) \? x \in E_n)\\ 
&\leq  \EE(\mu_n(B(x,r)))\PP(x\in E_n)^{-1}\\
&\leq 2^dr^d p_n^{-1} \leq 2^dp_N^{-1}r^{s},
\end{aligned}
\end{equation*}
the last inequality holds by $p_n>p_N $ and $0<r<1$.

Case 2.  $n \geq N$.  There will appear three subcases depending on the size of $r$.

Subcase 1. $r>r_N\sqrt{d}$. In this case, we have
\begin{equation*}
\begin{aligned}
\EE&(\mu_n(B(x,r) ) \big\? x \in E_n)\leq 1\\
&= r^{-s}r^s\leq (r_N\sqrt{d})^{-s} r^s.
\end{aligned}
\end{equation*}

Subcase 2. $r \leq r_n\sqrt{d}$. In this case we have 
\begin{equation}\label{eq:one}
\begin{aligned}
\mu_{n}(B(x,r))&=\int \text{\bf 1}_{E_n\cap B(x,r)}(y)p_n^{-1}dy \\
&\leq 2^{d}r^{d}p_n^{-1} \leq 2^{d} r^{d}r_{n}^{s-d}\\
&\leq 2^d(\sqrt{d})^{d-s}r ^{s}.
\end{aligned}
\end{equation}
Since this holds for any $n\in \NN$, we have
\[
 \mathbb{E} \left( \mu_{n}\left(B(x,r) \right) \big\? x \in E_{n} \right) \leq 2^d(\sqrt{d})^{d-s}r ^{s}.
\]

Subcase 3. $\sqrt{d}r_n <r\leq r_N \sqrt{d}$. Let $\I=\I(B(x,r),k+1)$ be the collection of cubes of $\D_{k+1}$ which intersects $B(x,r)$. By a volume argument there exists a positive  constant $C_1$ such that 
\[
\# \I \leq C_ 1 \left(\frac{r}{r_{k+1}} \right)^d.
\]
Note that for $Q\in \D_{k+1}$ and $x \notin Q$ we have 
\[
\PP(Q\subset E_{k+1}, x \in E_n) \leq \frac{N_{k+1}}{M_{k+1}^d}p_n,
\]
and hence 
\[
\PP(Q \subset E_{k+1}| x \in E_n) \leq \frac{N_{k+1}}{M_{k+1}^d}.
\]

Combining these with $P_k\geq r_k^{-s}$, we have
\begin{equation}\label{eq:two}
\begin{aligned}
&\mathbb{E} ( \mu_{n}(B(x,r)) \big\? x \in E_{n})
\leq \sum _{Q \in \I  } \mathbb{E} ( \mu_{n}(Q) \big\? x \in E_{n})\\
&=  \sum _{\substack {Q \in \I \\ x \notin Q } } \mathbb{E} (  \mu_{n}(Q)  \big\? x \in E_{n}) +  \sum _{\substack {Q \in \I \\ x\in Q  } } \mathbb{E} (  \mu_{n}(Q) \big\? x \in E_{n})\\
&\leq \# \I \,\frac{N_{k+1}}{M_{k+1}^d}P_{k+1}^{-1}+2^dP_{k+1}^{-1}\\
&
\leq C_1(\sqrt{d})^{d-s}r^{s}+2^{d}(\sqrt{d})^{-s}r^{s}\\
&\leq Cr^{s}.
\end{aligned}
\end{equation}

We fix a large constant $C$ such that all the above estimates hold. Thus we obtain the estimate \eqref{eq:expectmeasure}. 

Note that for any open set $O\subset [0,1]^{d}\times [0,1]^{d}$, we have (see \cite[Chapter 1]{Mattila1995})
\[
\mu\times \mu(O)\leq \liminf_{n\rightarrow \infty} \mu_{n}\times \mu_{n}(O).
\]

It follows that (let $B(x,r)$ be an open ball)
\begin{equation}
\begin{aligned}
&\int \mu(B(x,r) )d\mu(x) \\
&=\int \int \text{ \bf 1}_{\{(x,y): \?x-y\? < r\}} d\mu(x)d\mu(y) \\
&\leq \liminf_{n\rightarrow \infty} \int \int \text{ \bf 1}_{\{(x,y): \?x-y\? < r\}} d\mu_n(x)d\mu_n(y)\\
&= \liminf_{n\rightarrow \infty} \int \mu_n(B(x,r) )d\mu_n(x). \\
\end{aligned}
\end{equation}
 
Applying Fatou's lemma and \eqref{eq:expectmeasure}, we have
\begin{equation}
\begin{aligned}
&\EE  \left(\int \mu(B(x,r) )d\mu(x) \right)\\
&\leq \liminf_{n\rightarrow \infty} \EE \left(\int \mu_n(B(x,r) )d\mu_n(x) \right)\\
&= \liminf_{n\rightarrow \infty} \int_{[0,1]^d} p_n^{-1}\EE \left(\mu_n \left(B \left(x,r \right) \right) \textbf{1}_{E_n}(x) \right)dx \\
&= \liminf_{n\rightarrow \infty} \int_{[0,1]^d} \EE \left(\mu_n \left(B \left(x,r \right) \right) \big\?  x \in E_n \right)dx\\
&\leq C r^s.
\end{aligned}
\end{equation}
Thus we finish the proof.
\end{proof}

\begin{proof}[Proof of Theorem \ref{thm:localdimension} (3)]
For the lower bound, let $\varepsilon>0, s>0$ with $s+\varepsilon<s_1$. Note that for this $s$, by Lemma \ref{lem:massdistributionr} there is constant $C$ such that
the estimate \eqref{eq:expectmeasurefuthermore} holds. Let $\ell_j= 2^{-j}$ for $j\in \NN$. Then 
\begin{equation}\label{eq:inter}
\begin{aligned}
\EE &\left(\int \sum^\infty_{j=1} \ell_j^{-s}\mu(B(x,\ell_j) )d\mu(x) \right) \\
&= \sum^\infty_{j=1} \ell_j^{-s} \EE \left(\int \mu(B(x,\ell_j) )d\mu(x) \right)\\
&\leq C \sum^\infty_{j=1} \ell_j^{-s} \ell_j^{s+\varepsilon}<\infty
\end{aligned}.
\end{equation}
Thus we obtain that a.s. 
\[
\int \sum^\infty_{j=1}\ell_j^{-s}\mu(B(x,\ell_j) )d\mu(x)<\infty,
\]
and hence for $\mu$-a.e. $x$ 
\[
 \sum^\infty_{j=1} \ell_j^{-s}\mu(B(x,\ell_j))<\infty.
\]
Combining this with our choice  $\ell_j=2^{-j}$, we obtain $\ldim(\mu,x) \geq s.$  Since this holds for any $s<s_1,$ we have a.s. $\ldim(\mu,x)\geq s_1$ for $\mu$-a.e. $x$. Thus we finish the proof.
\end{proof}

\subsection{Almost sure upper local dimension}
Let $\ell_k=r_k/N_{k+1}^{\frac{1}{d}}, k\in \NN$. Applying the similar arguments to Lemma \ref{lem:massdistributionr}, we have the following  result.

\begin{lemma}\label{lem:massdistributiondelta}
For any $0<s<s_2$, there exists $C$ and a subsequence $\{\ell_{k_j}\}_{j\geq 1} \subset \{\ell_k\}_{k\geq 1}$, such that 
\begin{equation}
\EE \left(\mu_n(B(x, \ell_{k_j})) \big\? x \in E_n \right) \leq C \ell_{k_j}^{s}, \,\, j\in \NN.
\end{equation}
Furthermore we have
\begin{equation}
\EE\left( \int \mu(B(x, \ell_{k_j})) d \mu(x) \right)\leq C \ell_{k_j}^{s}, \,\, j\in \NN.
\end{equation}
\end{lemma}
\begin{proof}[Proof sketch]
For any $s<s_2$, there exists a subsequence $\{\ell_{k_j}\}_{j\geq 1} \subset \{\ell_k\}_{k\geq 1}$
such that $P_{k_{j}} \geq \ell_{k_j}^{-s}$ for all $j\in \NN$.

For each $j\in \NN$, let $\ell_{k_j}$ be the $r$ in the proof of Lemma \ref{lem:massdistributionr}. By the choice of $\{\ell_{k_j}\}_{j\geq 1}$, it is sufficient to consider Subcase 2 and Subcase 3 in the proof of Lemma \ref{lem:massdistributionr}. Moreover we use the estimate $P_{k_j}\geq \ell_{k_j}^{-s}$ at the estimates \eqref{eq:one} and \eqref{eq:two}. Thus we complete the proof.
\end{proof}

\begin{proof}[Proof of Theorem \ref{thm:localdimension} (4)] Lemma \ref{lem:lp} and Theorem \ref{thm:dimension} (2) imply that  for any $E \in \Omega$,
\[
\udim(\mu,x)\leq s_2
\]
holds for $\mu$-almost every $x\in E$. 

For the lower bound. Suppose $s_2>0$. Let $\varepsilon>0, s>0$ with $s+\varepsilon<s_2$. Applying Lemma \ref{lem:massdistributiondelta} and 
the same argument as in the estimate \eqref{eq:inter}, we obtain 
\begin{equation}
\begin{aligned}
\EE \Big(\int &\sum^\infty_{j=1}\ell_{k_j}^{-s}\mu(B(x,\ell_{k_j}) )d\mu(x) \Big) \\
&\leq C \sum^\infty_{j=1} \ell_{k_j}^{\varepsilon} \leq C \sum^\infty_{j=1} 2^{-k_{j}\varepsilon}\\
&< \infty
\end{aligned}.
\end{equation}
By the same argument as in the proof of Theorem \ref{thm:localdimension} (3), we complete the proof.
\end{proof}

\section{Hitting probabilities}\label{section:hitting}



In this section, we study the hitting probabilities of random Cantor sets in $\Omega(M,N)$. Note that the Hausdorff dimension of any $E\in \Omega$ is $\log N/ \log M=:s$. The methods which we use in the following proof are mainly from \cite[Chapter 8]{Falconer1997}, \cite{Peres} (first-Moment and second-Moment methods) and \cite{Shmerkin2}. 

Before we give the proof, we first show the following heuristic calculation. For $F\subset [0,1]^{d},$ define 
\[
F_n=\{Q\in \D_n: Q \cap F\neq \emptyset\}.
\]
Suppose $\#F_n$ roughly equals $ M^{n\alpha}$. We simply denote it as $\#F_n\sim  M^{n\alpha}$. Observe that 
\[
\EE(\# (F_n \cap E_n)) \sim M^{n\alpha} \left(\frac{N}{M^{d}}\right)^{n}=M^{(\alpha+s-d)n}.
\]
Therefore Theorem \ref{thm:hitting} should follows from the  relationships between $\alpha$ and $d-s$.

\begin{proof}[Proof of Theorem \ref{thm:hitting} (1)]  Recall that $\dim_HF =\alpha$ and $\alpha+s<d$. Applying the equivalent definition of Hausdorff dimension (\cite[Chapter 2.4]{Falconer2003}), we have that for any $\varepsilon>0$, there exists a sequence of interior disjoint cubes $\{Q_i\}_{i\in\NN} \subset \D$, such that $F \subset \bigcup_{i=1}^{\infty} Q_i$ and (see )
\begin{equation}\label{eq:firstmoment}
\sum_{i=1}^{\infty} \?Q_i\?^{d-s}<\varepsilon.
\end{equation}
Recall that $\?Q\?$ is the diameter of $Q$. For any $Q\in \D_n, n\in \NN$, we have
\begin{equation}
\begin{aligned}\label{eq:nnn}
\PP(Q \cap E \neq \emptyset)&\leq \PP( \text{ there exists } Q'\in E_n  \text{ with } Q' \cap Q \neq \emptyset) \\
&\leq 3^{d} (NM^{-d})^{n}=3^{d}M^{(s-d)n}\leq 3^{d} \?Q\?^{d-s}.
\end{aligned}
\end{equation}
Here we used the condition $N=M^{s}$. Observe that 
\[
(E \cap F \neq \emptyset) \subset \bigcup^{\infty}_{i=1} (E \cap Q_i \neq \emptyset).
\]
Combining this with the estimates \eqref{eq:firstmoment} and \eqref{eq:nnn}, we obtain
\begin{equation}
\begin{aligned}
\PP(E \cap F \neq& \emptyset)\leq \sum_{i=1}^{\infty}\PP(E \cap Q_i \neq \emptyset)\\
& \leq 3^d\sum_{i=1}^{\infty} |Q_i|^{d-s} < 3^d\varepsilon.
\end{aligned}
\end{equation}
We complete the proof by the arbitrary choice of  $\varepsilon$.
\end{proof}

\begin{proof}[Proof of Theorem \ref{thm:hitting} (2)] Let $\varepsilon>0$ such that $0< 2\varepsilon<\alpha+s-d$. Since $\dim_H F=\alpha$, by \cite[Theorem 4.10]{Falconer2003} there exists a compact subset $K\subset F$ such that $\dim_H K >\alpha-\varepsilon$. Furthermore,  by \cite[Theorem 4.13]{Falconer2003} there exists a probability measure $\lambda$ on $K$ such that for all $0<\beta<\alpha-\varepsilon$,  
\begin{equation}\label{eq:L^2bounded}
\E_{\beta}(\lambda):=\int\int d(x,y)^{-\beta}d\lambda(x)d\lambda(y) <\infty.
\end{equation}

For each $n \in \NN$, defining 
\[
K_n=\{Q\in \D_n^{*}: Q \cap K \neq \emptyset\}
\]
where $\D_n^{*}$ denotes the modification of $\D_n$ such that the elements of $\D_n^{*}$ form a partition of $[0,1]^{d}$. Roughly speaking, $\D_n^{*}$ denotes the collection of $M^{dn}$ ``half close half open cubes'' cubes with side length $M^{-n}$ such that any two distinct cubes are disjoint.

Let
\[
K_n=\{Q \in \D^{*}_n: Q \cap K \neq \emptyset\}
\]
(We may consider $K_n$ as a subset of $[0,1]^{d}$ for convenience of notation). For $\omega\in \Omega$, define the random set
\[
K_n^\omega= \{Q \in K_n: Q \subset E_n^\omega\}.
\]

Let $p:=N/M^{d}$, define the random measure
\begin{equation}\label{eq:martingalemeasure}
\nu_n^{\omega} = p^{-n} \lambda \big\?_{K^\omega_n}
\end{equation}
where $\lambda \big\?_{K^\omega_n}$ is the measure $\lambda$ restricted to $K^\omega_n$. Let 
\[
K^{\omega}= \bigcap^{\infty}_{n=1}K_n^{\omega}.
\]
Since $K$ is a compact set, we obtain that for any $\omega$,  
\begin{equation}
\label{eq:i}
K^{\omega}\subset K\cap E^{\omega}\subset F.
\end{equation}
In the following we intend to show that $\nu^{\omega}(K^{\omega})>0$ with positive probability, where $\nu^{\omega}$ is the weak limit measure of $\nu_n^{\omega}$.

The random sets $\{K^{\omega}_m\}_{1\leq m\leq n}$ give rise to an increasing filtration of $\sigma$-algebras $\mathcal{F}_n$.  For any $Q\in \D_n$, we have 
\[
\EE(\lambda(Q \cap K_{n+1}) \big| Q \in K_n)=p\lambda(Q)=p\lambda(Q\cap K_n)
\]
and 
\[
\EE(\lambda(Q \cap K_{n+1}) \big| Q \notin K_n)=0.
\]
Therefore $\EE(\lambda(Q \cap K_{n+1}^{\omega}) \big| \mathcal{F}_{n})= p \lambda(Q \cap K_n^{\omega}).$ In fact this estimates holds for any $Q\in \D^{*}_k, k\in \NN.$ It follows that
\begin{equation*}
\begin{aligned}
\EE(\nu^{\omega}_{n+1}(Q) \big| \mathcal{F}_n)&=p^{-n-1}\EE(\lambda (Q\cap K^{\omega}_{n+1}) \big| \mathcal{F}_n)\\
&= p^{-n}\lambda_n(Q \cap K_n^{\omega})=\nu_n^{\omega}(Q).
\end{aligned}
\end{equation*} 

Thus the sequence $\{\nu_n(Q), \mathcal{F}_n\}_{n\in \NN}$ is a martingale sequence. Applying the same argument as in \cite[Lemma 8.7]{Falconer1997}, we see that almost surely $\nu_n^{\omega}$ weakly converges to a measure $\nu^{\omega}$. Furthermore, applying Lemma \ref{lem:twopoint} and the condition \eqref{eq:L^2bounded} we obtain
\begin{equation*}
\begin{aligned}
&\EE((\nu_n([0,1]^d))^2)
=p^{-2n}\EE(\lambda(K_n)^2)\\
&=p^{-2n}\EE(\int \int \textbf{1}_{K_n \times K_n}(x,y) d\lambda(x) d\lambda(y))\\
&\leq C \int \int d(x,y)^{s-d-\varepsilon}d\lambda(x) d\lambda(y)<\infty.
\end{aligned}
\end{equation*}
It means that $\{\nu_n([0,1]^d)\}_{n\in\NN}$ is an $L^2$-bounded martingale. Thus by \cite[Corollary 8.4]{Falconer1997} we obtain that
\[
\EE(\nu([0,1]^d))=\EE(\nu_1([0,1]^d))=1,
\]
and hence $\nu^{\omega}([0,1]^d)>0$ with positive probability. Note that for any $\omega \in \Omega$, we have $\nu^\omega ([0,1] \backslash K^\omega)=0.$ It follows that 
$\nu^\omega(K^\omega)>0$ with positive probability. By the inclusion \eqref{eq:i} we complete the proof.
\end{proof}



\begin{proof}[Proof of Theorem \ref{thm:hitting} (3)] Let $\varepsilon>0$ such that $0< 2\varepsilon<\alpha+s-d$ and $t= \alpha+s-d-2\varepsilon$. We use the same notations as in the previous proof. 

Applying Fatou's lemma, Fubini theorem, and Lemma \ref{lem:twopoint}, we obtain 
\begin{equation}
\begin{aligned}
\EE  & \left(  \int  \int d(x,y)^{-t} d\nu(x) d\nu(y) \right) \\
&\leq\liminf_{n \rightarrow \infty}
\EE \left( \int \int d(x,y)^{-t}d\nu_n(x)d\nu_n(y) \right) \\
&=\liminf_{n\rightarrow \infty} \EE \left( \int \int d(x,y)^{t}p^{-2n}\textbf{1}_{K_n\times K_n}(x,y)d\lambda(x)d\lambda(y) \right) \\
&\leq C \int \int d(x,y)^{-t}d(x,y)^{s-d-\varepsilon}d\lambda(x)d\lambda(y) \\
&\leq C \int \int d(x,y)^{-\alpha+\varepsilon}d\lambda(x)d\lambda(y)< \infty.
\end{aligned}.
\end{equation}
The last inequality holds by the choice of $\lambda$, see estimate \eqref{eq:L^2bounded}. Recall that  $\nu^{\omega}(K^{\omega})>0$ with positive probability. As before this implies that 
\[
\dim_H (K^\omega)\geq \alpha-\varepsilon
\]
with positive probability. By the arbitrary choice of $\varepsilon$, we complete the proof.
\end{proof}

\begin{remark}  
Applying the similar argument to \cite[Chapter 7]{Peres}, we show a different  proof from above  for Theorem \ref{thm:hitting} (2) in the following.

\begin{proof}[Proof Sketch.]
For any $\varepsilon>0,$ there exists a compact subset $K \subset F$, such that $\dim_H K =\alpha-\varepsilon.$ We choose small $\varepsilon$
satisfies 
\begin{equation}\label{eq:2}
 \alpha+s>d+2\varepsilon.
\end{equation}
Recalling $E=\bigcap_{n=1}^{\infty}E_n$. Since $K$ is a compact set, we have 
\[
(E \cap K \neq \emptyset)=\bigcap^{\infty}_{n= 1} (E_n \cap K \neq \emptyset).
\]
Observe that the events $(E_{n} \cap K \neq \emptyset)$ is monotone decrease, hence we have
\begin{equation}\label{eq:ff}
\PP(E \cap K \neq \emptyset)=\lim_{n\rightarrow \infty}\PP (E_n \cap K \neq \emptyset).
\end{equation}

For each $n \in \NN$, defining 
\[
K_n=\{Q\in \D_n^{*}: Q \cap K \neq \emptyset\}
\]
where $\D_n^{*}$ denotes the modification of $\D_n$ such that the elements of $\D_n^{*}$ form a partition of $[0,1]^{d}$. Roughly speaking, $\D_n^{*}$ denotes the collection of $M^{dn}$ ``half close half open cubes'' cubes with side length $M^{-n}$ such that any two distinct cubes are disjoint.

Let $\lambda$ be a probability measure on $K$ such that for any  $0<\beta<\alpha-\varepsilon$,  
\begin{equation}\label{eq:en}
\E_{\beta}(\lambda):=\int\int d(x,y)^{-\beta}d\lambda(x)d\lambda(y) <\infty.
\end{equation}
Let $p=N/M^{d}$, defining
\[
Y_n=\sum_{Q \in K_n} p^{-n}\textbf{1}_{E_n}(Q) \lambda (Q), ~ n\in \NN
\]
where $\textbf{1}_{E_n}(Q)=1$ when $Q \subset E_n$, otherwise equal zero. For any $Q\in K_n$, we have $\PP(Q\subset E_n)=p^{n}.$ It follows that
\begin{equation}\label{eq:111}
\EE(Y_n)=\lambda(K_n)=1, ~ n\in \NN. 
\end{equation} 
Note that for any $n\in \NN$,
\begin{equation}\label{eq:fact}
  (Y_n >0) \subset (E_n \cap K \neq \emptyset).
\end{equation}

Observe that there exists a positive constant $C_1=C_1(d)$ such that for any $Q, Q' \in K_n, n\in \NN$, and  $x\in Q, x'\in Q'$, we have
\[
\PP(Q \subset E_n, Q' \subset E_n) \leq C_1\PP(x\in E_n, x'\in E_n).
\]
Note that the equality holds when $x$ and $x'$ are  interior point of $Q$ and $Q'$ respectively.
Applying Lemma \ref{lem:twopoint}, the conditions \eqref{eq:2} and \eqref{eq:en}, we obtain
\begin{equation}
\begin{aligned}\label{eq:YY}
\EE(Y_n^2)=&\sum_{Q \in K_n}\sum_{Q' \in K_n}p^{-2n}\lambda(Q)\lambda(Q')\PP(Q\subset  E_n, Q' \subset E_n)\\
&\leq C_1\sum_{Q \in K_n}\sum_{Q' \in K_n}p^{-2n} \int_{Q}\int_{Q'} \PP(x\in E_n, x'\in E_n) d\lambda(x)d\lambda(x')\\
&\leq C_1 C_2 \sum_{Q \in K_n}\sum_{Q' \in K_n}\int_{Q}\int_{Q'}d(x,x')^{s-d-\varepsilon}d\lambda(x)d\lambda(x')\\
&= C_1C_2 \E_{d-s+\varepsilon}(\lambda)<\infty.
\end{aligned}
\end{equation}
Here the constant $C_2$ comes from Lemma \ref{lem:twopoint}.  

By the Cauchy-Schwarz inequality, we obtain
\[
\EE(Y_n)^{2}=\EE(Y_n\textbf{1}_{(Y_n>0)})^{2}
\leq \EE(Y_n^{2})\PP(Y_n>0),
\]
and hence (Paley-Zygmund inequality)
\begin{equation}\label{eq:zgmend}
\PP(Y_n>0)\geq \dfrac{\EE(Y_n)^{2}}{\EE(Y_n^{2})}.
\end{equation}
Combining this with estimates \eqref{eq:111} and \eqref{eq:YY}, we obtain
\[
\PP(Y_n >0) \geq \frac{\EE(Y_n)^2}{\EE(Y_n^2)} \geq \frac{1}{C_1C_2 \E_{d-s+\varepsilon}(\lambda)}:=\delta >0.
\]
Applying the estimates \eqref{eq:ff} and \eqref{eq:fact}, we obtain
\begin{equation*}
\begin{aligned}
\PP(E\cap F\neq \emptyset)&\geq \PP(E \cap K \neq \emptyset)\\
&=\lim_{n\rightarrow \infty}\PP(E_n\cap K \neq \emptyset)\\
&\geq \liminf_{n\rightarrow \infty}\PP(Y_n >0)\geq \delta.
\end{aligned}
\end{equation*}
Thus  we complete the proof.
\end{proof}


\end{remark}

\section{Further results and questions}\label{section:remarks}

\subsection{Some examples for exceptional sets}\label{section:example} Here we present some examples of exceptional sets for the almost sure type results in the case $d=1$ (i.e. any element of $\Omega$ is a subset of $[0,1]$). For $\{n_k\}_{k\geq 1} \subset \NN$, we consider the space $\Omega=\Omega(3^{n_k}, 2^{n_k})$. In fact our examples will always looks like $\Omega(3^{n_k}, 2^{n_k})$, but the sequences $\{n_k\}$ are different in different examples. It is clear that for any $\{n_k\}_{k\geq 1} \subset \NN$ the classic Cantor ternary set $C \in \Omega$, and it is well known that 
\begin{equation}\label{eq:assouad2}
\dim_H C=\dim_A C= \frac{\log 2}{\log 3}.
\end{equation}
For convenience, let $s_k=\sum^k_{j=1}n_j$.

\begin{example}
Let $n_k/s_k \rightarrow 1$ as $k\rightarrow \infty$, then there exists $E \in \Omega$ such that $\ldim_B E =0$.
\end{example}
\begin{proof}
Note that for any $\{n_k\}_{k\geq 1} \subset \NN$,  Theorem \ref{thm:dimension} (3) claims that almost surely 
\[
\dim_H E= \ldim_B E= \frac{\log 2}{\log 3}.
\]
While Theorem \ref{thm:typical} (1) implies that for a typical $E\in \Omega$,  $\dim_B E =t^{*}=0$. However, we show an concrete example in the following for clearness. For $n_1$, we divide $[0,1]$ into $3^{n_1}$ interior disjoint $3^{n_1}$-adic closed intervals and choose $2^{n_1}$ closed intervals of them from the left part of $[0,1]$. They are interior disjoint and their union is $[0, 2^{n_1}3^{-n_1}]$. Let $E_1$ be the collection of these $2^{n_1}$ intervals. 
Given $E_k$, the collection of $2^{s_k}$ closed intervals with the same length $3^{-s_k}$. For every  interval $I \in E_k$, we divide it into $3^{n_{k+1}}$  interior disjoint $3^{s_{k}}$-adic closed intervals and choose $2^{n_{k+1}}$ closed intervals of them from the left part of $I$ (see Figure \ref{figure5}), and let $E_{k+1}$ be the union of the chosen closed intervals. Let $E= \bigcap_{k\geq 1} E_k$. Note that 
$n_k/s_k \rightarrow 1$ implies that $s_k/n_{k+1}\rightarrow0.$ For every $k \in \NN$, we have
\[
N(E, r_{k+1}N_{k+1})\leq P_k,
\]
and hence
\[
\frac{\log P_k}{-\log r_{k+1}N_{k+1}}=\frac{s_k \log 2}{s_k \log 3+ n_{k+1}\log (3/2)}\rightarrow 0.
\]
It follows that $\ldim_B E=0$. Thus we complete the proof.
\end{proof}

\begin{example}
Let $n_k/s_k \rightarrow 1$ as $k\rightarrow \infty$, then almost surely 
\[
\dim_P E=\udim_B E =1,
\]
and hence the Cantor set  is an exceptional set for Theorem \ref{thm:dimension} (4).
\end{example}
\begin{proof}
By a straight calculation, we have
\[
\frac{\log P_{k+1}}{-\log (r_k/N_{k+1})}=\frac{s_{k+1}\log 2}{s_k\log 3+n_{k+1}\log 2} \rightarrow 1 \text{ as } k\rightarrow \infty.
\]
The claim follows by Theorem \ref{thm:dimension} (4) and \eqref{eq:assouad2}.
\end{proof}

\begin{figure}
\medskip
\resizebox{0,8\textwidth}{!}{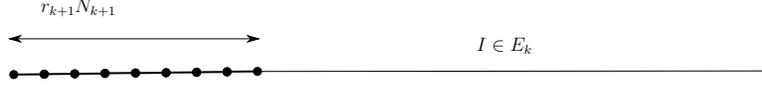}
\caption{There are $8$ subintervals of $I$ which belong to $E_{k+1}$, and all of them accumulate at the left part of $I$. We can think this as the one dimensional version of Figure \ref{figure3}.}
\label{figure5} 
\end{figure}

\begin{example}
Let $n_k \rightarrow \infty$. Then Theorem \ref{thm:dimension} (5) claims that almost surely $\dim_A E= 1$. Thus the Cantor set $C$ is an exceptional set. 
\end{example}

\subsection{Typical local dimension}

Recall that for any $E\in \Omega$, there is a natural measure $\mu$ on $E$. We can also study the typical local dimensions for these natural measures.

\begin{proposition}
(1) For a typical $E\in\Omega$, and \textbf{all} $x\in E$, we have
\[
\ldim(\mu, x)= t^{*}. 
\]

(2) For a typical $E\in\Omega$, and \textbf{all} $x\in E$, we have
\[
\udim(\mu, x)\geq s_2. 
\]
\end{proposition}

\begin{proof}
The claim (1) follows from the Remark \ref{re:lowlocal} and the proof of Theorem \ref{thm:dimension} (1). The claim (2)  follows immediately from the proof of the Theorem \ref{thm:typical} (2).
\end{proof}

We do not know whether we can obtain equality in the above claim (2).

\subsection{Normal numbers}
It is clear that the Cantor ternary set does not contain any normal numbers, but things are different when we add randomness. We have the following result for our random Cantor sets under the natural measure $\mu$. For the definition of normal numbers and further results, see \cite{Bugeaud}.

\begin{proposition}\label{Pro:normal}
Almost surely for $E\in \Omega$, we have that $\mu$-almost all $x \in E$ is a normal number.
\end{proposition}

This follows by Borel's normal numbers theorem and the following Lemma. Recall that Borel's normal number theorem claims 
that almost every (with respect to Lebesgue measure) real numbers are normal. The following Lemma (observation) is due to Pablo Shmerkin.

\begin{lemma}\label{pro:measureargument}
Let $F\subset [0,1]^d$ with $\mathcal{L}(F)=0$. Then almost surely $\mu(F)=0.$ 
\end{lemma}
\begin{proof}
Let $\varepsilon>0$, then there is an open set $U\supset F$ with $\mathcal{L}(U)<\varepsilon$.  Note that $\mu(U)\leq \liminf_{n\rightarrow \infty}\mu_n(U)$, see \cite[Theorem 1.24]{Mattila1995}. Applying Fubini's theorem we obtain 
\[
\EE (\mu_n(U))=\EE (\int \textbf{1}_{(U\cap E_n)}(x) p_n^{-1}dx) = \mathcal{L}(U).
\]
Combining these with Fatou's lemma, we have
\[
\EE(\mu(F))\leq \EE(\mu(U))\leq \liminf_{n\rightarrow \infty}\EE (\mu_n(U)) \leq \mathcal{L}(U) <\varepsilon. 
\]
By the arbitrary choice of $\varepsilon,$ we finish the proof.
\end{proof}

\subsection{Tube null sets}
A set $E \subset \RR^d (d \geq 2)$ is called tube null if for any $\varepsilon> 0$, there exist countable many tubes $\{T_i\}$ covering $E$ and $\sum_i w(T_i)^{d-1} < \varepsilon$. Here a tube $T$ with width $w= w(T) > 0$ is the $w/2$- neighborhood of some line in $\mathbb{R}^d$. We refer to \cite{Carbery} for the background and more details on tube null sets. In \cite{Shmerkin1}, the following result is proved.

\begin{proposition}
If $\sup_{k\in \NN} M_k <\infty$ and the almost sure Hausdorff dimension is larger than $d-1$, then almost surely $E$ is not tube null.
\end{proposition}

It is natural to ask that how about the case $\sup_{k\in \NN} M_k =\infty$. Another interesting question is that what will happen if there is no randomness. For instance, what happens for  the self-similar sets of $\Omega(M,N)$, that is the elements of $\Omega(M,N)$ we take the same position for the chosen subcubes in every step during our construction. For the self-similar sets, see \cite[Chapter 9]{Falconer2003}. 

\begin{question}
Is every self-similar set of $\Omega(M,N)$ tube null (exclude the trivial one with $N=M^{d}$)? 
\end{question}

Note that the classical Marstrand-Mattila projection theorem (see e.g \cite{Falconer2003,Mattila1995}) implies that any set $E\in \RR^{d}$ with $\dim_H E<d-1$ is tube null, see \cite[Proposition 7]{Carbery}. Thus it is sufficient to consider the self-similar set of $\Omega(M,N)$ with Hausdorff dimension larger or equal $d-1$ for above question.

We can also consider which kind of self-similar set or self-affine sets are tube null. In \cite{Harangi}, the author proved that the Koch snowflake curve is tube null. In fact we can apply the similar arguments to \cite{Harangi} to obtain that  the Sierpi\'nski  triangle is tube null also, we omit the details here. For self-affine sets and Bedford-McMullen carpets, see \cite[Chapter 9]{Falconer2003}.

\begin{question}
Is every Bedford-McMullen carpet tube null (exclude the trivial carpet which is the unit cube)?
\end{question}

\subsection{Lower dimension}

The \emph{lower dimension} can be considered as the dual of Assouad dimension. 
It is defined as follows: 
\begin{align*}
\dim_{L}E  = \sup\Big\{ s \geq 0:\exists ~C>0 &\textrm{ s.t. }  \forall ~0<r<R <\sqrt{d},\\ 
& \inf_{x\in E}\mathcal{N}(E\cap B(x,R),r)\geq C\left(R/r\right)^{s}
\Big\}.
\end{align*}
The lower dimension was introduced by Larman, see \cite{Larman}. For the recent works on the Lower dimension, we refer to \cite{Fraser2014} and references therein. For our random Cantor sets, if $\{N_k\}$ is bounded then we have the dual result for the lower dimension. 

\begin{proposition}
If $\{N_k\}$ is bound, then for any $E\in \Omega$ we have 
\[
\dim_L E= \liminf_{k\rightarrow \infty} \inf_{n\in \NN} \frac{\log P(n,n+k) }{ -\log r(n,n+k)}.
\]
\end{proposition}
\begin{proof}[Proof Sketch.]
If $\{M_n\}$ is bound, then we obtain the result by the similar argument as in the proof for Assouad dimension.

For the case $\{M_n\}$ is unbound. Observe that any set $E\in \Omega$ has lower dimension zero. Thus it is sufficient to show that the formula also give the zero value. This follows from the fact that for any $k\in \NN,$  
\[
\inf_{n\in \NN} \frac{\log P(n,n+k) }{ -\log r(n,n+k)}=0.
\]
Thus we complete the proof.
\end{proof}
We do not know the general result  for the lower dimension of these random Cantor sets when $\{N_k\}$ is unbound. We show two examples in the following with special sequence $M_k, N_k$. 

\begin{example}
If there exists a subsequence $\{n_k\}\subset \NN$ such that $M_{n_k}\nearrow \infty$ and  $\liminf_{n_k\rightarrow \infty}\frac{\log N_{n_k}}{\log M_{n_k}}=0$, then any element of $\Omega(M_n,N_n)$ has lower dimension zero.
\end{example}
\begin{proof} 
Let $E\in \Omega$. For any $\varepsilon>0$, there exists $N$ such that $n_k\geq N$ implies $\log N_{n_k} / \log M_{n_k} <\varepsilon.$ Note that there exists $C>0$ which depends on $d$ only such that for any $x\in E$,
\begin{equation*}
\begin{aligned}
\mathcal{N}(E&\cap B(x, r_{n_k-1}), r_{n_k})\\
&\leq CN_{n_k}\leq C M_{n_k}^{\varepsilon}=C \left(\frac{r_{n_k-1}}{r_{n_k}}\right)^{\varepsilon}.
\end{aligned}
\end{equation*}
By the condition that $M_{n_k}\nearrow \infty$, we obtain that $\dim_L E \leq \varepsilon$, and hence $\dim_LE=0$ by the arbitrary choice of $\varepsilon$.
\end{proof}
 
This example responds 
an interesting fact of lower dimension that is  if a set $E$ has isolate point then $E$ has lower dimension zero.

\begin{example}
Let $M_n=2^{n}$ and $N_n=2^{nd}-1$. Then any element of $\Omega(M_n,N_n)$ has lower dimension $d$.
\end{example}
\begin{proof}[Proof Sketch.]
Let $E\in \Omega$. Note that there exist positive constants $C_1, C_2$ such that for any $x\in E, 0<R < \sqrt{d}$, 
\[
C_1 R^{d}\leq \mathcal{L}(E\cap B(x,R)) \leq C_2 R^{d}.
\]
Hence there exists $C_3$ such that for any $x\in E, 0<r<R < \sqrt{d}$, 
\[
\mathcal{N}(E\cap B(x,R), r) \geq C_3 \left(\frac{R}{r}\right)^{d}.
\]
Thus the claim follows by the fact that any set of $\RR^{d}$ has lower dimension less or equal than $d$.
\end{proof}

\noindent\textbf{Acknowledgements.} 
I appreciate to Ville Suomala for many helpful discussions go though this work. I also thank Bing Li, 
Pablo Shmerkin, and Shengyou Wen for many valuable discussions.


\end{document}